\documentclass[a4paper,reqno,twoside,11pt]{amsart}

\usepackage{a4wide}
\usepackage{amssymb}
\usepackage{latexsym}
\usepackage{amsmath}
\usepackage{color,soul}
\usepackage{amssymb,dsfont,tikz-cd,hyperref}

\theoremstyle{plain}
\newtheorem{proposition}{\sc Proposition}[section]
\newtheorem{theorem}[proposition]{\sc Theorem}

\newtheorem{corollary}[proposition]{\sc Corollary}

\theoremstyle{definition}
\newtheorem{definition}[proposition]{\sc Definition}
\newtheorem{example}[proposition]{\sc Example}

\theoremstyle{remark}

\numberwithin{equation}{section}
\begin{document}

\title[A state-space approach to quantum permutations]
 {A state-space approach to quantum permutations}

\author{J.P. McCarthy}

\address{Department of Mathematics, Munster Technological University, Cork, Ireland.}
\email{jeremiah.mccarthy@mtu.ie}

\begin{abstract}
In this exposition of quantum permutation groups, an alternative to the `Gelfand picture' of compact quantum groups is proposed. This point of view is inspired by algebraic quantum mechanics and interprets states of an algebra of continuous functions
on a quantum permutation group as quantum permutations. This interpretation allows talk of an \emph{element} of a   quantum permutation group, and allows a clear understanding of the difference between deterministic, random, and quantum permutations. The interpretation is illustrated by various  quantum permutation group phenomena.
\end{abstract}

\maketitle

\section*{Introduction}
In the world of noncommutative mathematics, $\mathrm{C}^*$-algebraic compact quantum groups as defined by Woronowicz \cite{woro2} have been around  for about 35 years. No different to many other topics in modern mathematics, it is difficult to have a happy introduction to the world of these quantum groups without a certain level of mathematical maturity, not only in terms of technical knowledge, but also the marrying of this technical knowledge with an easiness with abstraction. A rudimentary acquaintance with quantum physics won't hinder either.   When the conceptual leap is  made, from classical compact groups with commutative algebras of continuous functions, to compact \emph{quantum} groups with \emph{noncommutative} algebras of continuous functions, and with the understanding
for the first time that  a compact quantum group with a noncommutative algebra of continuous functions is ``virtual'', not a set at all, let alone a group, compact quantum groups can be enjoyed as beautiful, intriguing, and mysterious entities.

\bigskip

With for example the quantum Peter--Weyl theorem and the Haar state, compact quantum groups are beautiful in how elegantly their theory generalises that of compact groups. With the emergence of ``infinite''  quantum generalisations $S_N^+$ of $S_N$ for $N\geq 4$, and non-Kac compact quantum groups where the ``inverse'' appears non-involutive, intriguing in how the theory differs. The theory has matured greatly; today a lot of new results for compact quantum groups are written in the more technically demanding language of the locally compact quantum groups of Kustermans and Vaes \cite{KV}, and the theory is in the foothills of having applications to quantum information by way of the theory of quantum automorphism groups of finite graphs (see for example \cite{lupini}).

\bigskip

Through all of this, the mystery of the virtual nature of quantum groups remains. The first strike against this conceptual barrier is to deploy a formal notation, the \emph{Gelfand picture}, that makes sense whenever the object in question is classical. Rather than talking about a quantum group $A$, or, slightly better, an algebra of continuous functions $A$ on a compact quantum group, talk instead about an algebra of continuous functions, $C(\mathbb{G})$, on a compact quantum group $\mathbb{G}$. Or the algebra of regular functions $\mathcal{O}(\mathbb{G})$, or the algebra of (essentially) bounded measurable functions $L^{\infty}(\mathbb{G})$.

\bigskip

Is it possible to interpret compact quantum groups as being like compact groups beyond the correct but staid: ``they are generalisations of compact groups in the sense that compact quantum groups with commutative algebras of continuous functions can be identified with compact groups''? Do quantum group theorists hide  behind their beautiful and intriguing results about quantum groups an even deeper intuition for  these objects? If these deeper intuitions and interpretations exist, the author has not seen them written down anywhere in any detail.

\bigskip

 It could be argued that, like the notation emanating from the Gelfand picture, good intuition certainly helps the beginner, and possibly the expert too. The aim of this work is to present a good intuition/interpretation for the class of compact quantum groups known as \emph{quantum permutation groups}.   Unconventionally, it identifies a \emph{set}, the set of states of the algebra of functions, as the set of quantum permutations/elements of the quantum group.

\bigskip

This \emph{Gelfand--Birkhoff picture} requires a  leap to be made before ever `going quantum'. Pick up a fresh deck of $N$ cards in some known order and ``randomly'' shuffle the deck. The shuffle is distributed according to some probability $\nu$ in the set of probabilities on $S_{N}$, $M_p(S_N)$. Without turning over the cards, i.e. making some measurements, it can not be said exactly what permutation acted on the deck. The leap here is to not just consider as permutations the \emph{deterministic permutations} in $S_N$, but also the \emph{random permutations} in $M_p(S_N)$ (which includes via the Dirac measures the deterministic permutations). Bilinearly extending the group law to $M_p(S_N)$, gives the \emph{random group law}, the convolution
$$(\nu_2\star \nu_1)(\{\sigma\})=\sum_{\tau\in S_N}\nu_2(\{\sigma\tau^{-1}\})\nu_1(\{\tau\}).$$ The Dirac measure concentrated at the identity is an identity for the random group law. Furthermore the inverse can be extended to a map $\operatorname{inv}:M_p(S_N)\to M_p(S_N)$, which gives the inverse of a Dirac measure for the random group law. Precisely because the map $\operatorname{inv}$ is not an inverse on the whole of $M_p(S_N)$, the set of random permutations   does not form a group, but it is nonetheless a monoid whose elements can be well-interpreted, understood, and studied in their own right. Note that where $F(S_N)$ is the algebra of complex-valued functions on $S_N$, $M_p(S_N)$ is the subset of positive functionals of norm one on $F(S_N)$: the set of \emph{states} of $F(S_N)$.

\bigskip

Once this leap is made, that the elements of $M_p(S_N)$ can be studied as ``permutations'', it isn't so difficult to leap to quantum permutation groups $S_N^+$, where a state on an algebra $C(S_N^+)$ defining the quantum permutation group can be interpreted as a ``permutation'', a \emph{quantum permutation}. What makes the interpretation cogent is the choice to interpret the generators $u_{ij}$ of $C(S_N^+)$ as relating to the states of $C(S_N^+)$ precisely as the generators $\mathds{1}_{j\to i}(\sigma)=\delta_{i,\sigma(j)}$ of the complex-valued functions on $S_N$ relate to the states of $F(S_N)$. That is, as Bernoulli random variables with distribution
$$\mathbb{P}[\mathds{1}_{j\to i}=1\,|\,\nu]=\nu(\mathds{1}_{j\to i})\qquad (\nu\in M_p(S_N)).$$
Furthermore, given a random permutation $\nu\in M_p(S_N)$, this is precisely the probability that it maps $j\to i$. Use the notation thus:
\begin{equation}\mathbb{P}[\nu(j)=i]:=\nu(\mathds{1}_{j\to i});\label{eqp}\end{equation}
with some probability, the random permutation $\nu$ maps $j\to i$.

 \bigskip

 A big question here: does an exposition of an interpretation comprise mathematics? It can be claimed that it is at least \emph{of} mathematics; quoting William Thurston \cite{WT}:
\begin{quote}
\emph{This question brings to the fore something that is fundamental and pervasive:
that what we} [mathematicians] \emph{are doing is finding ways for people to understand and think about
mathematics}.
\end{quote}

 Please note that no claim of originality is made: the work is exposition of well-established theory from a certain point of view. Neither does the work comprise a survey (comprehensive or otherwise).  Those interested in learning more about compact quantum groups in general can consult the original papers of Woronowicz \cite{woro1,woro2}, with exposition/survey well served by the lecture notes (see the web) of Banica, Franz--Skalski--So{\l}tan, Freslon, Skalski, Weber, and Vergnioux. Overarching references are Timmermann \cite{Timm} and Neshveyev--Tuset \cite{NS}. For those interested in quantum permutations specifically, see the original paper of Wang \cite{Wang}, the survey of Banica--Bichon--Collins \cite{BBC1}, and the tome of Banica \cite{TeoTome}.   It would be remiss not to give a few references that experts have said came to mind when it was communicated that an expository piece on intuition/interpretation for quantum permutation groups was being worked on: see \cite{ATS,BEVW,CM,cleve,lupini,musto,RS,soltan}.
 
 \bigskip
 
 As a final introductory remark, with the help of the Borel functional calculus, the ideas contained here can also be extended to the case of the quantum orthogonal and unitary groups, $O_N^+$ and $U_N^+$. The approach here exploits an action $S_N^+\curvearrowright\{1,2,\dots,N\}$: doing the same for $O_N^+$ and $U_N^+$ exploits  actions on the real and complex spheres.

 \bigskip

The paper is organised as follows. In Section 1 the conventional Gelfand picture of $\mathrm{C}^*$-algebras is outlined together with a very brief overview of the theory of $\mathrm{C}^*$-algebraic compact quantum groups. Section 2 introduces the state-space-as-quantum-space \emph{Gelfand--Birkhoff picture}, and motivates this by introducing the rudiments of quantum probability and measurement, including the \emph{Born rule}, \emph{sequential measurement}, and \emph{wave function collapse}. In Section 3, following a layperson's motivation of what a quantum permutation should be, the quantum permutation groups of Wang and their subgroups are introduced, along with some basic properties. Section 4 starts with more focussed discussion of the Gelfand--Birkhoff picture, and makes this cogent by introducing the \emph{Birkhoff slice} (essentially the  extension of (\ref{eqp}) to \emph{quantum permutations}). This gives enough intuition to inspire the ``simplest yet'' proof of no quantum permutations on three symbols, as well as a lucid understanding of how deterministic and random permutations sit in a quantum permutation group. In Section 5 the convolution of states is defined as the quantum group law, and the counit and antipode understood on this level. The duals of discrete groups that are finitely-generated by elements of finite order, understood as \emph{abelian} with respect to the quantum group law, are studied in more depth as quantum permutation groups. Finally in Section 6 some intrigue: an exploration of some of the phenomena that occur once the commutative world of classical groups is left.

\section{Compact quantum groups}
In 1995, Alain Connes posed the question: \emph{What is the quantum automorphism group of a space}? For the case of finite spaces, this question was answered in 1998 by Shuzou Wang \cite{Wang}. There are two main ways of defining this quantum automorphism group but first some noncommutative terminology/philosophy is required.
\subsection{The Gelfand picture}
The prevailing point of view in the study of  compact quantum groups is to employ what could be called the \emph{Gelfand picture}. This starts with a categorical equivalence given by Gelfand's Theorem:
$$\text{compact Hausdorff spaces}\simeq (\text{unital commutative $\mathrm{C}^*$-algebras})^{\text{op}}.$$
Starting with a compact Hausdorff space $X$, the algebra of continuous functions on $X$, $C(X)$, is a unital commutative $\mathrm{C}^*$-algebra; and starting with a unital commutative $\mathrm{C}^*$-algebra $A$, the spectrum, $\Omega(A)$, the set of \emph{characters}, non-zero *-homomorphisms $A\to \mathbb{C}$, is a compact Hausdorff space such that $A\cong C(\Omega(A))$. Therefore a general unital commutative $\mathrm{C}^*$-algebra can be denoted $A=C(X)$. Inspired by this, one can define the category of `compact quantum spaces' as
$$\text{compact quantum spaces}:\simeq (\text{unital $\mathrm{C}^*$-algebras})^{\text{op}}.$$
In analogy with the commutative case, a general not-necessarily commutative unital $\mathrm{C}^*$-algebra can be denoted $A=C(\mathbb{X})$, the unit $I_A=:\mathds{1}_{\mathbb{X}}$, $C(\mathbb{X})$ be called an algebra of continuous functions on the quantum space $\mathbb{X}$, and  $\mathbb{X}$ referred to as the spectrum of $A$. However, in the Gelfand picture, $\mathbb{X}$ is not a set any more but a so-called \emph{virtual} object, only spoken about via its algebra of continuous functions.

\bigskip

Some basic knowledge of $\mathrm{C}^*$-algebras is required here. See \cite{Murph} for further details on the below. The set of states,  $\mathcal{S}(C(\mathbb{X}))$, is the set of positive linear functionals $C(\mathbb{X})\to \mathbb{C}$ of norm one.  A state $\varphi\in \mathcal{S}(C(\mathbb{X}))$ is \emph{pure} if it has the property  that whenever $\rho$ is a positive linear functional such that $\rho\leq \varphi$, necessarily there exists $t\in [0,1]$ such that $\rho=t\varphi$.  Otherwise $\varphi$ is \emph{mixed}. Elements of the form $g^*g\in C(\mathbb{X})$ are \emph{positive}, and for positive $f\in C(\mathbb{X})$ there exists a (pure) state $\varphi$ such $\varphi(f)=\|f\|$. Let $\pi(C(\mathbb{X}))\subseteq B(\mathsf{H})$ be a unital representation.  For non-zero $\xi\in\mathsf{H}$ and $\hat{\xi}=\xi/\|\xi\|$,
$$\varphi_\xi(f)=\langle\hat{\xi},\pi(f)\hat{\xi}\rangle,$$
defines a state on $C(\mathbb{X})$ called a \emph{vector state}. The \emph{GNS construction} $\pi_\varphi(C(\mathbb{X}))\subseteq B(\mathsf{H}_\varphi)$ gives norm one $\xi_\varphi\in \mathsf{H}_\varphi$ such that
\begin{equation}\varphi(f)=\langle \xi_\varphi,\pi_\varphi(f)\xi_\varphi\rangle.\label{Born}\end{equation}
Therefore  all states are vector states for some representation.

\bigskip

An element $p\in C(\mathbb{X})$ is a \emph{projection} if $p=p^*=p^2$. For $f\in C(\mathbb{X})$ define $|f|^2=f^*f$. If $p_1,p_2,\dots,p_n\in C(\mathbb{X})$ are projections,
\begin{equation}|p_np_{n-1}\cdots p_2p_1|^2=p_1p_2\cdots p_{n-1}p_np_{n-1}\cdots p_2p_1.\label{abs}\end{equation}
This work will consider (finite) \emph{partitions of unity}, (finite) sets of projections $\{p_i\}_{i=1}^n\subset C(\mathbb{X})$ such that
$$\sum_{i=1}^np_i=\mathds{1}_{\mathbb{X}}.$$
Necessarily elements of partitions of unity are pairwise orthogonal, $p_ip_j=\delta_{i,j}p_i$.
If $C(\mathbb{X})$ is finite dimensional it will be denoted by $F(\mathbb{X})$, the algebra of functions on a finite quantum space $\mathbb{X}$, in this case  isomorphic to a multi-matrix algebra:
$$F(\mathbb{X})\cong \bigoplus_{i=1}^mM_{N_i}(\mathbb{C}).$$
 If $F(\mathbb{X})$ is commutative, then $\mathbb{X}=:X$ is a finite set, $F(X)$ is the algebra of all complex valued functions on it, isomorphic to $\mathbb{C}^N$, equivalently the diagonal subalgebra of $M_{N}(\mathbb{C})$.

\subsection{Compact Quantum Groups}\label{Quantum Groups}
 In the well-established setting of $\mathrm{C}^*$-algebraic compact quantum groups, it is through the Gelfand picture that one speaks of a quantum group $\mathbb{G}$, through a unital noncommutative $\mathrm{C}^*$-algebra $A$ that is considered an algebra of continuous functions on it: $A=C(\mathbb{G})$. If $X$ and $Y$ are compact topological spaces, then, where $\otimes$ is the minimal tensor product:
$$C(X\times Y)\cong C(X)\otimes C(Y).$$
Let $S$ be a compact semigroup. Using the above isomorphism, the transpose of the continuous multiplication $m:S\times S\to S$ is a $*$-homomorphism, the \emph{comultiplication}:
$$\Delta:C(S)\to  C(S)\otimes C(S).$$
The associativity of the multiplication gives \emph{coassociativity} to the comultiplication:
\begin{equation}(\Delta\otimes I_{C(S)})\circ \Delta=(I_{C(S)}\otimes \Delta)\circ \Delta.\label{coass}\end{equation}
If $C(S)$ satisfies \emph{Baaj--Skandalis cancellation}
   $$\overline{\Delta(C(S))(\mathds{1}_S\otimes C(S))}=\overline{\Delta(C(S))(C(S)\otimes \mathds{1}_S)}=C(S)\otimes C(S);$$
  then $S$ has cancellation, and a compact semigroup with cancellation is a group. In this sense Baaj--Skandalis cancellation is a $C(S)$-analogue of cancellation. This inspires a definition which took its first form in Baaj and Skandalis \cite{BaS}, with refinements from Woronowicz \cite{woro2}, and finally Van Daele \cite{VaD}:

\begin{definition}
\textbf{An} \emph{algebra of continuous functions on a ($\mathrm{C}^*$-algebraic) compact quantum group} $\mathbb{G}$ is a unital $\mathrm{C}^*$-algebra $C(\mathbb{G})$ together with a unital $*$-morphism $\Delta:C(\mathbb{G})\to C(\mathbb{G})\otimes C(\mathbb{G})$ that satisfies coassociativity and \emph{Baaj--Skandalis cancellation}:
 $$\overline{\Delta(C(\mathbb{G}))(\mathds{1}_{\mathbb{G}}\otimes C(\mathbb{G}))}=\overline{\Delta(C(\mathbb{G}))(C(\mathbb{G})\otimes \mathds{1}_{\mathbb{G}})}=C(\mathbb{G})\otimes C(\mathbb{G}).$$
  If $C(\mathbb{G})$ is finite dimensional, $\mathbb{G}$ is said to be a \emph{finite quantum group}, and $F(\mathbb{G})$ written for $C(\mathbb{G})$.
\end{definition}
Algebras of continuous functions on compact quantum groups come with a dense Hopf $*$-algebra $\mathcal{O}(\mathbb{G})$ of \emph{regular functions.} Hopf $*$-algebras are $*$-algebras which satisfy axioms which are precisely  $F(G)$-analogues of the (finite) group axioms. The $F(G)$-analogue of the group law, is \emph{comultiplication} $\Delta:F(G)\to F(G)\otimes_{\text{alg.}} F(G)$; the $F(G)$-analogue of the (inclusion of the) identity is the \emph{counit} $\varepsilon:F(G)\to \mathbb{C}$, $\varepsilon(f):=\operatorname{ev}_e(f)=f(e)$; and the $F(G)$-analogue of the inverse is an antihomomorphism called the \emph{antipode}, $S:F(G)\to F(G)$, $Sf(\sigma)=f(\sigma^{-1})$. A most leisurely introduction to how these maps, and the $F(G)$-analogues of associativity (\emph{coassociativity}, (\ref{coass})), of the identity axiom (the \emph{counital property}), and of the inverse axiom (the \emph{antipodal property}), are $F(G)$-analogues of the (finite) group axioms is given in Section 1.1, \cite{DS}.

\bigskip

Note the stress on \emph{an}: the algebra of regular functions can have more than one completion. There is a maximal, \emph{universal} completion $C_u(\mathbb{G})$, and a minimal, \emph{reduced} completion $C_r(\mathbb{G})$. In this sense, a compact quantum group can be identified  with non-isomorphic algebras of continuous functions $C_\alpha(\mathbb{G})$ and $C_\beta(\mathbb{G})$ if their dense algebras of regular functions are isomorphic as Hopf $*$-algebras. Non-isomorphic algebras of continuous functions can arise for the dual of a discrete group $\Gamma$. The dual $\widehat{\Gamma}$ is a compact quantum group, with algebra of regular functions given by the group ring, $\mathcal{O}(\widehat{\Gamma}):=\mathbb{C}\Gamma$. All the completions of $\mathcal{O}(\widehat{\Gamma})$ are canonically isomorphic exactly when $\Gamma$ is amenable: when all the completions of the algebra of regular functions on a compact quantum group $\mathcal{O}(\mathbb{G})$ are canonically isomorphic, in particular $C_u(\mathbb{G})\cong C_r(\mathbb{G})$, the compact quantum group is said to be \emph{coamenable}.

\bigskip

What makes a compact quantum group a generalisation of a compact group? Consider a commutative algebra of functions on a compact quantum group, $A$. Gelfand's Theorem states $A\cong C(\Omega(A))$, the unital $*$-homomorphism $\Delta:C(\Omega(A))\to C(\Omega(A)\times \Omega(A))$ gives a continuous map $m:\Omega(A)\times \Omega(A)\to \Omega(A)$, the coassociativity of $\Delta$ gives associativity to $m$, and so $(\Omega(A),m)$ is a compact semigroup with Baaj--Skandalis cancellation, and so a group. That compact quantum groups with commutative algebras of continuous functions are in fact (classical) compact groups is Gelfand duality: that the virtual quantum objects are still studied through their algebra of functions is the essence of the Gelfand picture.

\bigskip

A particular class of compact quantum group, earlier defined by Woronowicz \cite{woro1}, is a \emph{compact matrix quantum group}.
\begin{definition}
If a compact quantum group $\mathbb{G}$ is such that
\begin{itemize}
\item $C(\mathbb{G})$ is generated  by the entries of a unitary matrix $u\in M_N(C(\mathbb{G}))$, and
\item $u$ and $u^t$ are invertible, and
\item $\Delta:C(\mathbb{G})\to C(\mathbb{G})\otimes C(\mathbb{G})$, $u_{ij}\mapsto \sum_{k=1}^Nu_{ik}\otimes u_{kj}$ is a $*$-homomorphism,
\end{itemize}
then $\mathbb{G}$ is a \emph{compact matrix quantum group} with \emph{fundamental representation} $u\in M_N(C(\mathbb{G}))$. \end{definition}
\begin{theorem}(Woronowicz)\label{commmat}
If a compact matrix quantum group $G$ has commutative algebra of functions $C(G)$, then $G$ is homeomorphic to a closed compact subgroup of the unitary group, $U_N$.\end{theorem}

\bigskip

The quantum groups studied in this work are all compact matrix quantum groups such that $u$, the \emph{fundamental representation}, is a \emph{magic unitary}. That is the rows and columns of $u$ are partitions of unity:
$$\sum_{k=1}^Nu_{ik}=\mathds{1}_{\mathbb{G}}=\sum_{k=1}^Nu_{kj}.$$
Such compact quantum groups are called \emph{quantum permutation groups}. There are finite quantum groups which are not quantum permutation groups \cite{BBN}.
\section{Quantum Mathematics} \label{somequantummech}
The preceding section outlines the conventional view of compact quantum groups. The Gelfand picture allows nominal talk of a compact quantum group as an object, but in general does not permit the consideration of an \emph{element} of a compact quantum group. Aspects of quantum mechanics can be used to inspire a way of doing this for quantum permutation groups.

\bigskip

In his book Weaver \cite{Weaver} states and argues the point that:
\begin{quote}
\emph{The fundamental idea of mathematical quantisation is that sets are replaced by Hilbert spaces... [and] the quantum version of a [real]-valued function on a set is a [self-adjoint] operator on a Hilbert space.}
\end{quote}
Weaver attributes the Hilbert-space-as-set point of view to Birkhoff and von Neumann \cite{BirkvonN}, and the operator-as-function point of view to Mackey \cite{Mackey}. In this picture, the elements of the projective version of a Hilbert space $P(\mathsf{H})$ form a quantum space, and the self-adjoint operators are random variables $P(\mathsf{H})\to \mathbb{R}$, with  the Born rule providing probability, and spectral projections providing wave function collapse. Call this the \emph{Birkhoff picture}.

\bigskip

Inspired by algebraic quantum mechanics \cite{Landsman}, and its descendent quantum probability (as seen in e.g. \cite{Maasen}, rather than the free probability of Voiculescu), this work will push on slightly, and instead define the compact quantum space associated to a $\mathrm{C}^*$-algebra to be the \emph{set} of \emph{states} on the algebra.  Call this the \emph{Gelfand--Birkhoff picture}.  This is a well worn path outside the field of compact quantum groups, see \cite{Landsman2,Landsman3} for discussion and further references.

\bigskip

 To illustrate, consider a state $\varphi$ on a unital $\mathrm{C}^*$-algebra $C(\mathbb{X})$, the self-adjoint elements of which are called \emph{observables}. Consider a projection, a \emph{Bernoulli observable}, $p\in C(\mathbb{X})$. Associated to $p$ are two events: $p=1$ given by $p^1:=p$, and $p=0$ given by $p^0:=\mathds{1}_{\mathbb{X}}-p$. The distribution of $p$ given the state $\varphi$ is given by (essentially the Born rule by (\ref{Born})):
$$\mathbb{P}[p=\theta\,|\,\varphi]:=\varphi(p^{\theta}).$$
If the event $p=\theta$ is non-null, $\mathbb{P}[p=\theta\,|\,\varphi]>0$, and the measurement of $\varphi$ with $p$ gives $p=\theta$,  the state transitions $\varphi\mapsto \widetilde{p^\theta}\varphi$ where for $f\in C(\mathbb{X})$:
$$\widetilde{p^{\theta}}\varphi(f):=\frac{\varphi(p^\theta fp^{\theta})}{\varphi(p^{\theta})}.$$
This is \emph{wave function collapse}, or \emph{state conditioning}.  If $\varphi_\xi$ is a vector state given by a representation $\pi(C(\mathbb{X}))\subseteq B(\mathsf{H}$), and $\varphi_{\xi}(p^{\theta})>0$ (so that  $\pi(p^{\theta})\neq 0$), then $\widetilde{p^\theta}\varphi_{\xi}$ is also a vector state, given by $\varphi_{\pi(p^{\theta})\xi}$. To see this use the fact that $\pi(p^{\theta})$ is a projection:
\begin{align*}
\widetilde{p^{\theta}}\varphi_{\xi}(f)&=\frac{\varphi_\xi(p^\theta fp^\theta)}{\varphi_\xi(p^\theta)}= \frac{\langle\xi,\pi(p^{\theta}fp^{\theta})\xi\rangle}{\langle\xi,\pi(p^{\theta})\xi\rangle}=\frac{\langle\pi(p^{\theta})\xi,\pi(f)\pi(p^{\theta})\xi\rangle}{\langle\pi(p^{\theta})\xi,\pi(p^{\theta})\xi\rangle}\\
&=\frac{\langle\pi(p^{\theta})\xi,\pi(f)\pi(p^{\theta})\xi\rangle}{\|\pi(p^{\theta})\xi\|^2}=\left\langle\frac{\pi(p^{\theta})\xi}{\|\pi(p^{\theta})\xi\|},\pi(f)\frac{\pi(p^{\theta})\xi}{\|\pi(p^{\theta})\xi\|}\right\rangle=\varphi_{\pi(p^{\theta})\xi}(f).
\end{align*}

Take another projection $q\in C(\mathbb{X})$. Suppose that the event $p=\theta_1$ has been observed so that the state is now $\widetilde{p^{\theta_1}}\varphi$. The probability that measurement \emph{now} produces $q=\theta_2$, and $\widetilde{p^{\theta_1}}\varphi\mapsto \widetilde{q^{\theta_2}}\widetilde{p^{\theta_1}}\varphi$, is:
$$\mathbb{P}[q=\theta_2\,|\,\widetilde{p^{\theta_1}}\varphi]:=\widetilde{p^{\theta_1}}\varphi(q^{\theta_2})=\frac{\varphi(p^{\theta^1}q^{\theta_2}p^{\theta_1})}{\varphi(p^{\theta_1})}.$$
Define now the event $\left([q=\theta_2]\succ [p=\theta_1]\,|\,\varphi\right)$, said `given the state $\varphi$, $q$ is measured to be $\theta_2$ \emph{after} $p$ is measured to be $\theta_1$'. Using the expression above a probability can be ascribed to this event:
\begin{align*}
\mathbb{P}\left[[q=\theta_2]\succ [p=\theta_1]\,|\,\varphi\right]&:= \mathbb{P}[q=\theta_2\,|\,\widetilde{p^{\theta_1}}(\varphi)]\cdot\mathbb{P}[p=\theta_1\,|\,\varphi]
\\&= \frac{\varphi(p^{\theta_1}q^{\theta_2}p^{\theta_1})}{\varphi(p^{\theta_1})}\cdot \varphi(p^{\theta_1})\cdot=\varphi(p^{\theta_1}q^{\theta_2}p^{\theta_1})=\varphi(|q^{\theta_2}p^{\theta_1}|^2).
\end{align*}
Inductively, for a finite sequence of projections $(p_i)_{i=1}^n$, and $\theta_i\in\{0,1\}$:
$$\mathbb{P}\left[[p_n=\theta_n]\succ\cdots \succ[p_1=\theta_1]\,|\,\varphi\right]=\varphi(|p_n^{\theta_n}\cdots p_1^{\theta_1}|^2).$$
It is worth noting that
\begin{equation}
\mathbb{P}[[p_2=\theta_2]\succ [p_1=\theta_1]|\varphi]\leq \mathbb{P}[p_1=\theta_1|\varphi],\label{check}
\end{equation}
so that in particular if $\mathbb{P}[[p_2=\theta_2]\succ [p_1=\theta_1]|\varphi]>0$ then $\mathbb{P}[p_1=\theta_1|\varphi]>0$.

\bigskip

In general, $pq\neq qp$ and so
$$\mathbb{P}\left[[q=\theta_2]\succ [p=\theta_1]\,|\,\varphi\right]\neq \mathbb{P}\left[[p=\theta_1]\succ [q=\theta_1]\,|\,\varphi\right],$$
and this is to be interpreted that $q$ and $p$ are not simultaneously observable. However the \emph{sequential projection measurement} $q\succ p$  is an `observable' in the sense that it is a random variable with values in $\{0,1\}^2$. Inductively, the sequential projection measurement $p_n\succ \cdots\succ p_1$ is a $\{0,1\}^n$-valued random variable.

\bigskip

 If $p$ and $q$ \emph{do} commute, then the distributions of $q\succ p$ and $p\succ q$ are equal in the sense that
$$\mathbb{P}\left[[q=\theta_2]\succ [p=\theta_1]\,|\,\varphi\right]=\varphi(|q^{\theta_2}p^{\theta_1}|^2)=\mathbb{P}\left[[p=\theta_1]\succ [q=\theta_2]\,|\,\varphi\right];$$
it doesn't matter what order they are measured in, the outputs of the measurements can be multiplied together, and this observable can be called $pq=qp$.

\bigskip

Measurement with a projection includes an \emph{a priori} distribution, and  wave function collapse: but these are not purely quantum mechanical phenomena, and occur also with measurements from a commutative subalgebra $C(X)\subseteq C(\mathbb{X})$. To illustrate, let $X=\{x_1,\dots,x_N\}$ and consider the diagonal subalgebra $F(X)$ of $F(\mathbb{X}):=M_N(\mathbb{C})$. Subsets $Y\subseteq X$ yield subspaces $F(Y)\subseteq F(X)$ together with projections $p_Y\in F(\mathbb{X})$. The \emph{a priori} distribution of $p_Y$ is:
$$\mathbb{P}[p_Y=\theta\,|\,\varphi]=\varphi(p_Y^\theta),$$
and, conditional on $p_Y=\theta$, there is wave function collapse to $\widetilde{p_Y^\theta}\varphi\in \mathcal{S}(F(X))$.

\bigskip

A defining difference between classical and quantum measurement is the quantum phenomenon of projection observables that cannot be simultaneously measured in the sense that $(p\succ q)\neq (q\succ p)$. In classical measurement, all projection observables can be simultaneously measured, and this implies that while classical measurement \emph{can} disturb a mixed state, the effects are purely probabilistic, capturing a decrease in uncertainty about the state. In the finite, classical case of $F(X)$, measurement with an appropriate sequence of classical projection measurements results in collapse to a pure state $\operatorname{ev}_{x_i}\in \mathcal{S}(F(X))$. Pure states of the diagonal subalgebra are invariant under wave function collapse: further measurement does \emph{not} disturb the state. In contrast, for any state on $F(\mathbb{X}):=M_N(\mathbb{C})$ there is a projection $p\in F(\mathbb{X})$ that can disturb it, and so collapse to complete certainty is impossible.

\bigskip

Sequential measurement of finite spectrum observables $f\in C(\mathbb{X})$ can also be considered. Through the inverse of the isometric Gelfand--Naimark *-isomorphism, a finite spectrum observable $f\in C(\mathbb{G})$ has a spectral decomposition $f=\sum_{i=1}^{|\sigma(f)|}f_i\,p^{f_i}$, that defines a partition of unity $\{p^{f_i}\}_{i=1,\dots, |\sigma(f)|}\subset C(\mathbb{X})$, and
\begin{equation}\mathbb{P}[f=f_i\,|\,\varphi]=\varphi(p^{f_i}).\label{spectral}\end{equation}
Furthermore the expectation of $f$ can be defined:
$$\mathbb{E}[f\,|\,\varphi]:=\varphi(f).$$

For continuous spectrum observables, by passing to the enveloping von Neumann algebra $C(\mathbb{X})^{**}\cong \pi_U(C(\mathbb{X}))''$, which will be denoted $\ell^{\infty}(\mathbb{X})$, and taking the normal extension of $\varphi$ to a state $\omega_\varphi$ on $\ell^{\infty}(\mathbb{X})$, Borel functional calculus can be used to measure, for example, if $f$ is in some Borel subset of its spectrum, via the projection $\mathds{1}_{S}(f)\in \ell^{\infty}(\mathbb{X})$, so that $$\mathbb{P}[f\in S\,|\,\varphi]:=\mathbb{P}[\mathds{1}_{S}(f)=1\,|\,\varphi]:=\omega_\varphi(\mathds{1}_{S}(f)).$$
If $f\in S$ is non-null, for wave function collapse, embed functions $f\in C(\mathbb{X})$ via $\imath:C(\mathbb{X})\hookrightarrow \ell^{\infty}(\mathbb{X})$, and the state transitions to $\varphi\mapsto \widetilde{\mathds{1}_{S}(f)}\varphi$ defined by:
$$\widetilde{\mathds{1}_{S}(f)}\varphi(g)=\frac{\omega_\varphi(\mathds{1}_{S}(f)\imath(g)\mathds{1}_{S}(f))}{\omega_\varphi(\mathds{1}_{S}(f))}.$$
Although not considered in this work, for not-necessarily finite spectrum observables $(f_i)_{i=1}^n$, the distribution of the sequential measurement $f_n\succ \cdots \succ f_1$ could be defined for Borel sets $S_i\subseteq \sigma(f_i)$:
$$\mathbb{P}[(f_n\succ \cdots\succ f_1)\in (S_n,\cdots,S_i)|\varphi]=\omega_\varphi(|\mathds{1}_{S_n}(f_n)\cdots \mathds{1}_{S_1}(f_1)|^2).$$

\section{Quantum permutations}
Fresh decks of playing cards produced by e.g. the US Playing Card Company always come in the same original order:
$$A\spadesuit,\dots, K\spadesuit,A\clubsuit,\dots,K\clubsuit,A\diamondsuit,\dots,K\diamondsuit,A\heartsuit,\dots,K\heartsuit.$$
Respectively enumerate using $c:\{1,2,\dots, 52\}\to \{A\spadesuit,\dots,K\heartsuit\}$. The original order can be associated with the pure state   $\operatorname{ev}_e\in M_p(S_{52})$. After a suitably randomised shuffle, an active permutation, the deck will be in some unknown order given by a mixed state, a \emph{random permutation} $\nu\in M_p(S_{52})$, with the card in position $j$ moved to position $\nu(j)$. Suppose the card in position $i$ is turned over to reveal card $c(j)$. This observable, denoted $x^{-1}(i)\in F(S_{52})$, reveals that the random permutation sent $j$ to $i$. This observable has spectrum $\sigma(x^{-1}(i))=\{1,2,\dots,52\}$, and thus spectral decomposition
$$x^{-1}(i)=\sum_{k=1}^{52}k\,v_{ik},$$
with $\{v_{ik}\}_{k=1,\dots,52}$ a partition of unity.
The distribution of $x^{-1}(i)$ given the state $\nu\in M_p(S_{52})$ can be denoted
\begin{equation}\mathbb{P}[\nu^{-1}(i)=j]:=\mathbb{P}[x^{-1}(i)=j\,|\,\nu]=\nu(v_{ij}).\label{uij}\end{equation}
Each card $c(j)$ must be mapped \emph{somewhere} and so, for all $\nu\in M_p(S_{52})$
$$\sum_{k=1}^{52}\mathbb{P}[\nu^{-1}(k)=j]=\nu\left(\sum_{k=1}^{52}v_{kj}\right)=1,$$
this implies that $\{v_{kj}\}_{k=1,\dots,52}$ is also a partition of unity, giving another observable
$$x(j):=\sum_{k=1}^{52}k\,v_{kj},$$
and note that
$$\mathbb{P}[\nu(j)=i]:=\mathbb{P}[x(j)=i\,|\,\nu]=\nu(v_{ij})=\mathbb{P}[\nu^{-1}(i)=j].$$
The observable $x^{-1}(i)$ is measured by turning over the card in position $i$. How is $x(j)$ measured? Go back to the deck in the original order, turn card $c(j)$ face \emph{up}, and shuffle with $\nu$: the position of card $c(j)$ after the shuffle is $x(j)$.

\bigskip

Following the sequential measurement
$$x^{-1}(51)\succ \cdots\succ x^{-1}(2)\succ x^{-1}(1),$$
the random permutation will collapse to a (deterministic) permutation $\operatorname{ev}_{\sigma}\in M_p(S_{52})$. If the sequential measurement is paused, say at $\ell<51$ with
$$(x^{-1}(\ell)\succ \cdots\succ x^{-1}(2)\succ x^{-1}(1))=(j_\ell,\dots,j_2,j_1),$$
then the state has collapsed to
$$\nu_\ell:=\widetilde{v_{ij_{\ell}}}\cdots \widetilde{v_{ij_{2}}}\widetilde{v_{ij_{1}}}\nu.$$
Of course
$$\mathbb{P}[\nu_\ell(k)=j]:=\mathbb{P}[x^{-1}(k)=j_k\,|\,\nu_\ell]=\delta_{j,j_k},$$
that is once a card $c(k)$ is observed in the position $j_k$ once, that is determined once and for all.

\bigskip

 Note that $(v_{ij})_{i,j=1}^{52}\in M_{52}(F(S_{52}))$ is a \emph{magic unitary}.

\bigskip

There is no issue whatsoever talking about the set of random permutations, $M_p(S_N)$, nor an element of this set $\nu\in M_p(S_N)$. Inspired by the Gelfand--Birkhoff picture, imagine for a moment that the same can be done for \emph{quantum permutations}: imagine that there is a $\mathrm{C}^*$-algebra $C(S_N^+)$  such that the set of quantum permutations on $N$ symbols is given by the state space, and a quantum permutation is simply an element  of the state space.

\bigskip

What would make a permutation \emph{quantum}? In light of previous discussions perhaps what might make a permutation \emph{quantum} is that quantum versions of observables $x^{-1}(i)$ and $x(j)$ be \emph{not} simultaneously observable. This implies that, with a deck of cards shuffled with a quantum permutation, once the first card has been revealed, the observation of the second card might disturb the state in a such a way that non-classical events \emph{can} occur. What would be a non-classical event: turning over the first card to reveal the ace of hearts, then turning over the second card to reveal an ace of spaces, then turning over the first card again to find it is not longer the ace of hearts but the ace of diamonds:
$$\left[\varphi^{-1}(1)=c^{-1}(A\diamondsuit)\right]\succ \left[\varphi^{-1}(2)=c^{-1}(A\spadesuit)\right]\succ \left[\varphi^{-1}(1)=c^{-1}(A\heartsuit)\right].$$

\bigskip

With the deck in the original order, $x(j)$ would be measured by turning  card $c^{-1}(j)$ face \emph{up}, shuffling with $\varphi$, and noting the position of $c^{-1}(j)$ after the shuffle. Similarly $x^{-1}(i)$ would be observable by revealing the card in position $i$. What would not be permitted would be shuffling with more than one card face up, or revealing more than one card at once.

\bigskip

Given a quantum permutation $\varphi\in \mathcal{S}(C(S_N^+))$, similarly to before, the spectral decompositions of the observables $x^{-1}(i)$ and $x(j)$ should give a magic unitary, denoted now by $(u_{ij})_{i,j=1}^{52}\in M_{52}(C(S_N^+))$. Denote as before
$$\mathbb{P}[\varphi(j)=i]:=\mathbb{P}[x(j)=i\,|\,\varphi]=\varphi(u_{ij}).$$
The projective nature of wave function collapse, that conditional on $\varphi(j)=i$, $\varphi\mapsto \widetilde{u_{ij}}\varphi$, implies that
$$\mathbb{P}[[\varphi(j)=i]\succ[\varphi(j)=i]]=\mathbb{P}[\varphi(j)=i];$$
and also that the probability of observing $\varphi(j)=i$ after (just) observing $\varphi(j)=i$ is one.

\bigskip

In the sequel this will all be made mathematically precise.

\subsection{Wang's Quantum Permutation Groups}\label{Wang}
  In a survey article, Banica, Bichon and Collins \cite{BBC1} attribute to Brown \cite{Brown} the idea of taking a matrix group $G\subseteq U_N$, realising $C(G)$ as a universal commutative $\mathrm{C}^*$-algebra generated by the matrix coordinates $u_{ij}\in C(G)$ subject to some relations $R$, and then studying, if it exists, the noncommutative universal $\mathrm{C}^*$-algebra generated by abstract variables $u_{ij}$ subject to the same relations $R$. This procedure, later called \emph{liberation} in the context of compact quantum groups by Banica and Speicher \cite{BS}, was carried out by Wang to create quantum versions of the orthogonal and unitary groups, and later  quantum permutation groups.

\bigskip

Let $F(S_N)$ be the algebra of complex functions on $S_N$ with basis $\{\delta_\sigma\}_{\sigma\in S_N}$.  Define $\mathds{1}_{j\to i}\in F(S_N)$ by:
$$\mathds{1}_{j\to i}(\sigma):=\begin{cases}
    1, & \mbox{if } \sigma(j)=i, \\
    0, & \mbox{otherwise}.
  \end{cases}$$
Where $\Delta$ is the transpose of the group law $m:S_{N}\times S_{N}\to S_{N}$, so that $\Delta(f)=f\circ m$, and employing $F(S_N\times S_N)\cong F(S_N)\otimes_{\text{alg.}} F(S_N)$, note that $$\Delta(\mathds{1}_{j\to i})=\sum_{k=1}^N\mathds{1}_{k\to i}\otimes_{\text{alg.}} \mathds{1}_{j\to k}.$$
Furthermore
\begin{equation}\delta_\sigma=\prod_{j=1}^{N}\mathds{1}_{j\to \sigma(j)},\label{del}\end{equation}
and so the matrix $u=(\mathds{1}_{j\to i})_{i,j=1}^N$ is a unitary with inverse the transpose of $u$, whose entries generate $F(S_N)$. Therefore $F(S_N)$ is a commutative algebra of functions on a compact matrix quantum group. Furthermore the $\mathds{1}_{j\to i}$ are projections, and
$$\sum_{k=1}^N\mathds{1}_{k\to i}=\mathds{1}_{S_N}=\sum_{k=1}^N\mathds{1}_{j\to k}.$$
Therefore the matrix $u=(\mathds{1}_{j\to i})_{i,j=1}^N$ is a magic unitary. Indeed $F(S_N)$ has a presentation as a universal commutative $\mathrm{C}^*$-algebra:
$$F(S_N)\cong \mathrm{C}^*_{\text{comm}}(u_{ij}\,|\, u \text{ an $N\times N$ magic unitary}).$$
Following Wang \cite{Wang}, \emph{liberate} by considering the universal $\mathrm{C}^*$-algebra:
$$C(S_N^+):=\mathrm{C}^*(u_{ij}\,|\, u \text{ an $N\times N$ magic unitary}).$$
The universal property says that if $(v_{ij})_{i,j=1}^N$ is another $N\times N$ magic unitary, then $u_{ij}\mapsto v_{ij}$ is a $*$-homomorphism. It can be shown that
$$\left[\sum_{k=1}^Nu_{ik}\otimes u_{kj}\right]_{i,j=1}^N\in M_N(C(S_N^+)\otimes C(S_N^+))$$
is a magic unitary, and thus $\Delta(u_{ij})=\sum_{k=1}^Nu_{ik}\otimes u_{kj}$ is a $*$-homomorphism. It is straightforward to show that $\Delta$ is unital and coassociative, and so $C(S_N^+)$ is an algebra of continuous functions on a compact matrix quantum group, \emph{the} quantum permutation group on $N$ symbols.
\subsection{Quantum Permutation Groups}\label{subqpg} If $\mathbb{G}$ is a compact matrix quantum group whose fundamental representation is a  magic unitary, then the universal property gives $\pi: C(S_N^+)\to C(\mathbb{G})$ a surjective $*$-homomorphism that intertwines the comultiplication:
\begin{equation}\Delta_{C(\mathbb{G})}\circ \pi=(\pi\otimes \pi)\circ \Delta_{C(S_N^+)},\label{subgroup}\end{equation}
which is to say that $\mathbb{G}\subseteq S_N^+$, that is $\mathbb{G}$ is a quantum subgroup of $S_N^+$. Furthermore, if $\mathbb{G}\subseteq S_N^+$ by a comultiplication-intertwining surjective $*$-homomorphism $\pi_1:C(S_N^+)\to C(\mathbb{G})$, then $[\pi_1(u_{ij})]_{i,j=1}^N$ is a magic unitary that is a fundamental representation for $\mathbb{G}$.
\begin{definition}
  A \emph{quantum permutation group $\mathbb{G}$} is a compact matrix quantum group whose fundamental representation is a magic unitary. The notation
$$\mathbb{G}\subseteq S_N^+$$
implies a fixed fundamental magic representation $u\in M_N(C(\mathbb{G}))$, and subsequently $u_{ij}$ referring to a generator of $C(\mathbb{G})$ rather than of $C(S_N^+)$.
\end{definition}

\bigskip

For a quantum permutation group $\mathbb{G}\subseteq S_N^+$, on the algebra of regular functions $\mathcal{O}(\mathbb{G})$, also generated (as a $*$-algebra) by $u_{ij}\in \mathcal{O}(\mathbb{G})$, the comultiplication, counit, and antipodal maps are given by:
\begin{align}
  \Delta(u_{ij}) & =\sum_{k=1}^Nu_{ik}\otimes_{\text{alg.}} u_{kj}, \nonumber\\
  \varepsilon(u_{ij}) & =\delta_{i,j}, \label{counit}\\
  S(u_{ij}) & =u_{ji}.\nonumber
\end{align}
In general the counit does not extend to a character on a completion $C_\alpha(\mathbb{G})$ of $\mathcal{O}(\mathbb{G})$. The antipode satisfies $S^2=I_{\mathcal{O}(\mathbb{G})}$ so that $\mathcal{\mathbb{G}}$ is a \emph{Kac algebra}.

\bigskip

 The justification for calling $S_N^+$ \emph{the} quantum permutation group on $N$ symbols goes beyond the liberation of $F(S_N)$. Wang originally defined the (universal) quantum automorphism group of $\mathbb{C}^N$ (that leaves the counting measure invariant). This leads to the definition of $S_N^+$ given above. This work should further cement that $S_N^+$ is a quantum generalisation of $S_N$.
\begin{theorem}
For $N\leq 3$, $C(S_N^+)\cong F(S_N)$.
\end{theorem}
See Section \ref{Birk} for a new proof for $N=3$.
\begin{theorem}\label{infinitenc}
  For $N\geq 4$, $C(S_N^+)$ is noncommutative and infinite dimensional.
\end{theorem}
\begin{proof}
The standard argument for $N=4$ uses the universal $\mathrm{C}^*$-algebra generated by two projections (see \cite{BBC1}). To be slightly more concrete, consider the  infinite dihedral group
$$D_{\infty}:=\langle a,b\,|\,a^2=b^2=e\rangle. $$
The infinite dihedral group is amenable, which implies that the reduced and universal group $\mathrm{C}^*$-algebras coincide. Denote the (noncommutative) group ring by $\mathbb{C}D_{\infty}$ and $C(\widehat{D_\infty})$ \emph{the} $\mathrm{C}^*$-completion, which is in fact *-isomorphic to the universal $\mathrm{C}^*$-algebra generated by two projections \cite{Raeburn}. Together with $\Delta(g)=g\otimes g$, the dual $\widehat{D_{\infty}}$ is a compact matrix quantum group, with unit $\mathds{1}_{\widehat{D_{\infty}}}:=e$, and  fundamental representation $\operatorname{diag}(a,b)$.

\bigskip

In fact $\widehat{D_{\infty}}\subset S_4^+$: where $p:=(e+a)/2$ and $q:=(e+b)/2$, via $a=u_{11}-u_{21}$ (and similarly for $b$) the following is a magic fundamental representation for $\widehat{D_{\infty}}$:
\begin{equation}u:=\left[\begin{array}{cccc}
                          p & e-p & 0 & 0 \\
                          e-p & p & 0 & 0 \\
                          0 & 0 & q & e-q \\
                          0 & 0 & e-q & q
                        \end{array}\right].\label{pqmag}\end{equation}
Therefore $\widehat{D_\infty}\subset S_4^+$, and it follows that $C(S_4^+)$ is infinite dimensional and noncommutative.

\bigskip

 To extend to $N\geq 4$ use $\widehat{D_\infty}\subset S_{4+\ell}^+$ via $\operatorname{diag}(u,\mathds{1}_{\widehat{D_{\infty}}},\cdots,\mathds{1}_{\widehat{D_{\infty}}})$.
\end{proof}
Note also that replacing $a,b\in D_\infty$ with order two generators of $D_N$ shows that $\widehat{D_N}\subset S_4^+$ (exhibiting Th. 1.1 (9), \cite{4Pts}). Showing that $\widehat{D_3}\subset S_4^+$  is the easiest way of showing that $C(S_4^+)$ is noncommutative.

\bigskip

\begin{theorem}\label{coamenable}
  For $N\geq 5$, $S_N^+$ is not coamenable.
\end{theorem}
\begin{proof}
The standard argument that $S_N^+$ is not coamenable for $N\geq 5$ uses fusion rules \cite{genericcoaction}. However, in similar spirit to the (standard) proof of Theorem \ref{infinitenc}, using the fact that a compact subgroup of a coamenable compact quantum group is coamenable \cite{Salmi}, the exhibition of a non-coamenable subgroup of $S_5^+$ proves Theorem \ref{coamenable} for $N=5$ (the extension to $N>5$ follows in the same way as the extension of Theorem \ref{infinitenc} to $N>4$). Let $a$ and $b$ be the respective generators of $\mathbb{Z}_3\ast \mathbb{Z}_2$. Let $C(\widehat{\mathbb{Z}_3\ast \mathbb{Z}_2})$ be \emph{a} completion of $\mathbb{C}(\mathbb{Z}_3\ast \mathbb{Z}_2)$ to a compact quantum group. Where $\omega=e^{2\pi i/3}$, consider the following magic unitary:
$$u^{a}:=\frac{1}{3}\left[\begin{array}{ccc}
                                                 e+a+a^2 & e+\omega^2a+\omega a^2 & e+\omega a+\omega^2 a^2 \\
                                                 e+\omega a+\omega^2 a^2 & e+a+a^2 & e+\omega^2 a+\omega a^2 \\
                                                 e+\omega^2 a+\omega a^2 & e+\omega a+\omega^2 a^2 & e+a+a^2
                                               \end{array}\right]$$
Note that $a=u^{a}_{11}+\omega^2 u^{a}_{21}+\omega u^{a}_{31}$.
With the same notation $q=(e+b)/2$ as with the infinite dihedral group, except obviously with $e\in \mathbb{Z}_3\ast \mathbb{Z}_2$, let
$$u^{b}:=\left[\begin{array}{cc}
                                                 q & e-q \\
                                                 e-q & q
                                               \end{array}\right].$$

Consider the block magic unitary $\operatorname{diag}(u^a,u^b)$. This shows that $\widehat{\mathbb{Z}_3\ast \mathbb{Z}_2}\subset S^+_5$. The dual of a discrete group is coamenable if and only if the group is amenable; $\mathbb{Z}_3\ast \mathbb{Z}_2$ is not amenable \cite{Raeburn}, therefore its dual is not coamenable, and thus neither is $S^+_5$.
\end{proof}

Banica \cite{TeoTome} calls $\widehat{\mathbb{Z}_3\ast \mathbb{Z}_2}$ by Bichon's group dual subgroup of $S_5^+$. More on duals in Section \ref{duals}.

\bigskip

It is the case that $S_N\subseteq S_N^+$ is a quantum permutation group, known to be maximal for $N\leq 5$ \cite{Easiness}, but conjectured to be maximal for all $N\in\mathbb{N}$. One motivation for the current work is to perhaps provide some intuition to attack such a problem. See Section \ref{randomwalks} for more.

\section{The Gelfand-Birkhoff picture for quantum permutations}
\begin{definition}\label{defqp1}
Where $C(\mathbb{G})$ is an algebra of continuous functions on a quantum permutation group, \emph{a quantum permutation} is an element of the state space, $\mathcal{S}(C(\mathbb{G}))$
\end{definition}
 With the use of the \emph{Birkhoff slice} (Section \ref{Birk}) this statement will be made cogent, and, as will be seen in Section \ref{law}, there is a natural candidate for what should be considered a \emph{quantum group law} $\mathcal{S}(C(\mathbb{G}))\times \mathcal{S}(C(\mathbb{G}))\to \mathcal{S}(C(\mathbb{G}))$.

Returning to Gelfand's Theorem: as soon as an algebra of functions $C(\mathbb{G})$ is noncommutative, it is often remarked that it obviously cannot be the algebra of functions on a compact Hausdorff space (with the commutative pointwise multiplication). However the elements of $C(\mathbb{G})$ viewed through the lens of Kadison's function representation are affine functions on a compact Hausdorff space.

\bigskip

 Note firstly that $\mathcal{S}(C(\mathbb{G}))$ is a weak-* compact Hausdorff space \cite{Murph}, and  recall the embedding of $C(\mathbb{G})$ into the enveloping von Neumann algebra $\imath:C(\mathbb{G})\hookrightarrow \ell^{\infty}(\mathbb{G})$, $f\mapsto \imath(f)$:
$$\imath(f)(\rho):=\rho(f)\qquad (\rho\in C(\mathbb{G})^*).$$
Finally weak*-convergence of a net of states $\varphi_\lambda\to \varphi$, that for $f\in C(\mathbb{G})$
$$\varphi_\lambda(f)\to \varphi(f),$$
gives   continuity to $\imath(f)$:
$$\imath(f)(\varphi_\lambda)\to \imath(f)(\varphi).$$
The multiplication  $\imath(C(\mathbb{G}))\times \imath(C(\mathbb{G}))\to \imath(C(\mathbb{G}))$ is not the pointwise multiplication, but inherited from $C(\mathbb{G})$:
$$\imath(f)\imath(g)=\imath(fg)\neq \imath(gf)=\imath(g)\imath(f).$$

\bigskip

Through this lens the  elements of $C(\mathbb{G})$ are affine functions $\mathcal{S}(C(\mathbb{G}))\to \mathbb{C}$, that is for $\rho_1,\rho_2\in \mathcal{S}(C(\mathbb{G}))$, and $\lambda\in [0,1]$:
$$\imath(f)(\lambda\,\rho_1+(1-\lambda)\,\rho_2)=\lambda\,\imath(f)(\rho_1)+(1-\lambda)\,\imath(f)(\rho_2).$$
An element of $\imath(C(\mathbb{G}))\subsetneq C(\mathcal{S}(C(\mathbb{G})))$ is therefore completely determined by its values on the weak-$*$ closure of the set of pure states, the \emph{pure state space} $\mathcal{P}(C(\mathbb{G})):=\overline{PS(C(\mathbb{G}))}^{w^*}$.
 This is precisely how  a function on a finite group $f_0:G\to \mathbb{C}$ extends  to  an affine function on the set of \emph{random permutations} on $G$, $f_1:M_p(G)\to \mathbb{C}$:
$$f_1(\lambda\,\operatorname{ev}_{\sigma}+(1-\lambda)\,\operatorname{ev}_{\tau})=\lambda\,f_0(\sigma)+(1-\lambda)\,f_0(\tau).$$
While Gelfand's Theorem says, through the fact that the character space and pure state space coincide for unital commutative $\mathrm{C}^*$-algebras, that for a finite group $G$, the embedded $\imath(F(G))$ is the full algebra of continuous functions $F(\mathcal{P}(F(G)))$, in the noncommutative case, restricting even to $\mathcal{P}(C(\mathbb{G}))$, $\imath(C(\mathbb{G}))$ is a proper subset of $C(\mathcal{S}(C(\mathbb{G})))$, so, while tempting,  it is abuse of notation to define $\mathbb{G}:=\mathcal{P}(C(\mathbb{G}))$ as the compact Hausdorff space, and continue to use the $C(\mathbb{G})$ notation. If in the classical case the algebra of functions, $F(G)$ is understood as the algebra of affine functions on the random permutations $M_p(G)\to \mathbb{C}$, and an algebra of continuous functions on a quantum permutation group $C(\mathbb{G})$ is understood as an algebra of affine functions $\mathcal{S}(C(\mathbb{G}))\to \mathbb{C}$, then the relationship between $\mathcal{P}(C(\mathbb{G}))$ and $\mathcal{S}(C(\mathbb{G}))$ reflects in the quantum case the relationship between  $G$ and $M_p(G)$ in the classical case.

\bigskip

These analogies are well captured by the following schematic:
\begin{center}
\begin{tikzcd}
\text{Deterministic} \arrow[d, hook] \arrow[rr, hook] &  & \text{Random} \arrow[d, hook] \\
\text{Pure Quantum} \arrow[rr, hook]                  &  & \text{Quantum}
\end{tikzcd}
\end{center}
The objects on the left are pure states of $\mathrm{C}^*$-algebras; while the objects on the right are mixed states. The objects on top are states of commutative algebras; while the objects on the bottom are states of noncommutative algebras. The focus of this work is on the mixed states. This pair of dichotomies is discussed in \cite{HSO}.

\bigskip

Therefore with this focus, the set of quantum permutations will be denoted by $\mathbb{G}:=\mathcal{S}(C(\mathbb{G}))$,  \emph{a} quantum permutation written an element $\varphi\in \mathbb{G}$, but the $C(\mathbb{G})$ notation will be kept, but with the implicit understanding that it is a proper subset of $C(\mathcal{S}(C(\mathbb{G})))$ (not to mention the fact that for  non-coamenable $\mathbb{G}$ there are different $\mathrm{C}^*$-completions  of $\mathcal{O}(\mathbb{G})$, and thus different state spaces). 

\subsection{The Birkhoff Slice}\label{Birk}
Given a quantum permutation group $\mathbb{G}\subseteq S_N^+$, via the Gelfand--Birkhoff picture, an element $\varphi\in \mathbb{G}$ is a quantum permutation. In this picture, the projections $u_{ij}$ are Bernoulli observables. Make the following interpretation:
\begin{equation}\label{quantpermint}
  \mathbb{P}[\varphi(j)=i]:=\mathbb{P}[u_{ij}=1\,|\,\varphi]:=\varphi(u_{ij}).
\end{equation}
These probabilities can be collected in a matrix:
$$\Phi(\varphi)_{ij}:=\varphi(u_{ij}).$$
 That $u$ is a magic unitary implies that $\Phi(\varphi)$ is a
 doubly stochastic matrix, i.e. $\Phi(\varphi)$ is in the Birkhoff polytope $\mathcal{B}_N$, and call the map $\Phi:\mathbb{G}\to \mathcal{B}_N$ the \emph{Birkhoff slice}. It is called a slice as it only captures an ephemeral aspect of a quantum permutation; and is not injective.

  \bigskip

  In the case of compact matrix quantum groups, there is a natural generalisation of the Birkhoff slice, $\Phi:\mathcal{S}(C(\mathbb{G}))\to M_N(\mathbb{C})$. The restriction of this map to characters, an injective map $\Phi:\Omega(C(\mathbb{G}))\to M_N(\mathbb{C})$, has been studied previously. Immediately Woronowicz uses this map to prove Theorem \ref{commmat} \cite{woro1}. Kalantar and Neufang \cite{KN}, who associate to a (locally) compact quantum group $\mathbb{G}$, a (locally) compact classical group $\tilde{\mathbb{G}}$, use the map to show that in the case of a compact matrix quantum group, $\tilde{\mathbb{G}}$ is homeomorphic to $\Phi(\Omega(C(\mathbb{G})))$.

  \bigskip

 Assuming that $\mathbb{P}[\varphi(k)=\ell]\neq 0$, the quantum permutation $\varphi$ can be conditioned on $\varphi(k)=\ell$, and conditional probabilities   collected in a Birkhoff slice. Recall state conditioning:
\begin{align*}\widetilde{u_{\ell k}}(\varphi)&:=\frac{\varphi(u_{\ell k}\cdot u_{\ell k})}{\varphi(u_{\ell k})}
\\ \implies  \Phi(\widetilde{u_{\ell k}}(\varphi))&=\left[\frac{\varphi(u_{\ell k}u_{ij} u_{\ell k})}{\varphi(u_{\ell k})}\right]_{i,j=1}^N
\\ &=:\left[\mathbb{P}[\varphi(j)=i\,|\,\varphi(k)=\ell \right]_{i,j=1}^N
\end{align*}
\begin{proposition}\label{BS1}
Let $\mathbb{G}\subseteq S_N^+$. For $\varphi\in \mathbb{G}$, if $\varphi(u_{ij})$ is non-zero, the matrix $\Phi(\widetilde{u_{ij}}\varphi)$ has a one in the $(i,j)$-th entry.
\end{proposition} Proposition \ref{BS1}  implies that if e.g.
\begin{align*}
\Phi(\varphi)&=\left[\begin{array}{cccc}
                     \Phi(\varphi)_{11} & \Phi(\varphi)_{12} & \cdots & \Phi(\varphi)_{1N} \\
                     \Phi(\varphi)_{21} & \Phi(\varphi)_{22} & \cdots & \Phi(\varphi)_{2N} \\
                     \vdots & \vdots & \ddots & \vdots \\
                     \Phi(\varphi)_{N1} & \Phi(\varphi)_{N2} & \cdots & \Phi(\varphi)_{NN}
                   \end{array}\right]
                   \\ \implies \Phi(\widetilde{u_{22}}\varphi)&=\left[\begin{array}{ccccc}
                     \Phi(\widetilde{u_{22}}\varphi)_{11} & 0 & \Phi(\widetilde{u_{22}}\varphi)_{13} &  \cdots & \Phi(\widetilde{u_{22}}\varphi)_{1N} \\
                     0 & 1 & 0 &  \cdots & 0 \\
                     0 & 0 & \Phi(\widetilde{u_{22}}\varphi)_{33} & \cdots & \Phi(\widetilde{u_{22}}\varphi)_{3N} \\
                     \vdots & \vdots & \vdots & \ddots & \vdots \\
                     \Phi(\widetilde{u_{22}}\varphi)_{N1} & 0 & \Phi(\widetilde{u_{22}}\varphi)_{N3} & \cdots & \Phi(\widetilde{u_{22}}\varphi)_{NN}
                   \end{array}\right]\end{align*}

Inductively, assuming $\varphi(u_{i_nj_n}\cdots u_{i_1j_1})\neq 0$
$$\Phi(\widetilde{u_{i_nj_n}}\cdots \widetilde{u_{i_1j_1}}(\varphi))_{ij}=\mathbb{P}[\varphi(j)=i\,|\, [\varphi(j_n)=i_n]\succ\cdots\succ[\varphi(j_1)=i_1]].$$

Indeed
\begin{align*}
&\mathbb{P}[[\varphi(j)=i]\succ[\varphi(j_n)=i_n]\succ\cdots\succ[\varphi(j_1)=i_1]] =\varphi(|u_{ij}u_{i_nj_n}\cdots u_{i_1j_1}|^2)
\\ &=\Phi(\widetilde{u_{i_nj_n}}\cdots \widetilde{u_{i_1j_1}}(\varphi))_{ij}\cdot\Phi(\widetilde{u_{i_{n-1}j_{n-1}}}\cdots \widetilde{u_{i_1j_1}}(\varphi))_{i_nj_n}\cdots \Phi(\varphi)_{i_1j_1}.
\end{align*}
\begin{example}\emph{No quantum permutations on three symbols}
As a toy example of how the Gelfand-Birkhoff picture is a good intuition, consider the theorem that $S_3^+=S_3$. This is just to say that $C(S_3^+)$, the universal $\mathrm{C}^*$-algebra generated by a $3\times 3$ magic unitary $(u_{ij})_{i,j=1}^3$ is commutative. This was known by Wang \cite{Wang}, but Banica, Bichon, and Collins \cite{BBC1} describe the Fourier-type proof as ``quite tricky''. Lupini, Man\v{c}inska, and Roberson however give a more elementary proof \cite{lupini}.

\bigskip

By allowing talk of \emph{a} quantum permutation the Gelfand--Birkhoff picture  suggests \emph{why} there are no quantum permutations on three symbols. Without assuming $C(S_3^+)$ commutative, consider the observable
$$x(1)=u_{11}+2u_{21}+3u_{31},$$
which asks of a quantum permutation $\varphi\in S_3^+$ what it maps one to. Measure $\varphi$ with $x(1)$ and the denote the result by $\varphi(1)$. The intuition might be that as soon as  $\varphi(1)$ is known, $\varphi(2)$ and $\varphi(3)$ are entangled in the sense that measurement of $x(2)$ cannot be made without affecting $x(3)$ (but without affecting $x(1)$). This is only intuition:  it might still be possible to exhibit e.g. the non-classical event:
\begin{equation}[\varphi(3)=3]\succ [\varphi(2)=1]\succ [\varphi(1)=3];\label{imp}\end{equation}
but pausing before measuring $\varphi(2)$ allows the noting of a relationship between the events $[\varphi(2)=1]\succ [\varphi(1)=3]$ and $[\varphi(3)=2]\succ [\varphi(1)=3]$ that implies (\ref{imp}) cannot happen.

\bigskip

Suppose that $\varphi\in S_3^+$  and assume without loss of generality that  measuring $\varphi $ with $x(1)$ gives $\varphi(1)=3$ with non-zero probability $\varphi(u_{31})$, and the quantum permutation transitions to $\widetilde{u_{31}}\varphi\in S_3^+$.

 \bigskip

 Now consider,  using the fact that $u_{21}u_{31}=0=u_{32}u_{31}$, and the rows and columns of $u$ are partitions of unity:
\begin{align*}
u_{31}=(u_{12}+u_{22}+u_{32})u_{31}&=(u_{21}+u_{22}+u_{23})u_{31}
\\ \implies u_{12}u_{31}&=u_{23}u_{31}
\\ \implies u_{31}u_{12}&=u_{31}u_{23} ,
\end{align*}
by taking the adjoint of both sides. This implies that conditioning on $[\varphi(2)=1]\succ [\varphi(1)=3]$ is the same as conditioning on $[\varphi(3)=2]\succ [\varphi(1)=3]$:
\begin{equation}\widetilde{u_{12}}\widetilde{u_{31}}\varphi=\frac{\varphi(u_{31}u_{12}\cdot u_{12}u_{31})}{\varphi(u_{31}u_{12}u_{31})}=\frac{\varphi(u_{31}u_{23}\cdot u_{23}u_{31})}{\varphi(u_{31}u_{23}u_{31})}=\widetilde{u_{23}}\widetilde{u_{31}}\varphi\label{cond2}\end{equation}
Now
$$\Phi(\widetilde{u_{12}}\widetilde{u_{31}}(\varphi_\xi))=\left[\begin{array}{ccc}0 & 1 & 0 \\ \ast & 0 & \ast \\ \ast & 0 &\ast \end{array}\right].$$

Using (\ref{cond2})
\begin{align*}
\Phi(\widetilde{u_{12}}\widetilde{u_{31}}\varphi)_{23}&=\widetilde{u_{12}}\widetilde{u_{31}}\varphi(u_{23})=\widetilde{u_{23}}\widetilde{u_{31}}\varphi(u_{23})=\frac{\varphi(u_{31}u_{23}u_{23}u_{23}u_{31})}{\varphi(u_{31}u_{23}u_{31})}=1
\\ \implies \Phi(\widetilde{u_{12}}\widetilde{u_{31}}(\varphi))&=\left[\begin{array}{ccc}0 & 1 & 0 \\ 0 & 0 & 1 \\ \Phi(\widetilde{u_{12}}\widetilde{u_{31}}(\varphi))_{31} & 0 & 0 \end{array}\right],\end{align*}
and as $\Phi$ maps to doubly stochastic matrices, $\Phi(\widetilde{u_{12}}\widetilde{u_{31}}(\varphi))$ is equal to the permutation matrix $(132)$. This implies that $\widetilde{u_{12}}\widetilde{u_{31}}(\varphi)\in S_3^+$ is deterministic, that is its Birkhoff slice is a permutation matrix.

\bigskip

Not convinced this implies that $C(S_3^+)$ is commutative? Here is a proof inspired by the above.
\begin{theorem}
$C(S_3^+)$ is commutative.
\end{theorem}
\begin{proof}
It suffices to show that $u_{11}u_{22}=u_{22}u_{11}$ by showing:
$$u_{11}u_{22}=u_{11}u_{33}=u_{22}u_{33}=u_{22}u_{11}.$$
The first equality follows from:
$$u_{11}(u_{21}+u_{22}+{u_{23}})=u_{11}(u_{13}+u_{23}+u_{33}),$$
the second from
$$(u_{11}+u_{21}+u_{31})u_{33}=(u_{21}+u_{22}+u_{23})u_{33},$$
and the third from
$$u_{22}(u_{31}+u_{32}+u_{33})=u_{22}(u_{11}+u_{21}+{u_{31}}).$$
\end{proof}
\end{example}
\subsection{Deterministic Permutations}\label{Classical}
Let $\jmath:S_N\hookrightarrow M_N(\mathbb{C})$ be the embedding that sends a permutation to its permutation matrix.
A \emph{deterministic permutation in $\mathbb{G}$}  is a quantum permutation $\varphi\in\mathbb{G}$ such that $\Phi(\varphi)=\jmath(\sigma)$ for some $\sigma\in S_N$. In this case write $\varphi=\operatorname{ev}_\sigma$.
\begin{corollary}\label{characters}
A quantum permutation is deterministic if and only if it is a character.
\end{corollary}
\begin{proof}
Suppose that $\varphi=\operatorname{ev}_\sigma$ is deterministic. Consider the GNS representation $(\mathsf{H}_{\sigma},\pi_{\sigma},\xi_{\sigma})$ associated to $\varphi$. Note that by Proposition \ref{BS1}
$$\operatorname{ev}_\sigma (u_{ij})=\langle \xi_{\sigma},\pi_{\sigma}(u_{ij})(\xi_{\sigma})\rangle= 0 \text{ or }1.$$
Using the norm continuity of $\pi_{\sigma}$, this implies that for all $f\in C(\mathbb{G})$, there exists $f_\sigma\in \mathbb{C}$ such that $\pi_{\sigma}(f)(\xi_{\sigma})=f_{\sigma}\xi_{\sigma}$. Therefore
\begin{align*}
  \operatorname{ev}_{\sigma}(gf) & =\langle\xi_{\sigma},\pi_{\sigma}(gf)\xi_{\sigma}\rangle=\langle \xi_{\sigma},\pi_{\sigma}(g)\pi_{\sigma}(f)(\xi_{\sigma})\rangle \\
   & =f_{\sigma}\langle\xi_{\sigma},\pi_{\sigma}(g)\xi_{\sigma}\rangle=\operatorname{ev}_{\sigma}(g)\operatorname{ev}_{\sigma}(f).
\end{align*}
 On the other hand, by the homomorphism property of a character $\varphi\in\mathbb{G}$  $$\varphi(u_{ij})=\varphi(u_{ij}^2)=\varphi(u_{ij})^2\implies \varphi(u_{ij})=0\text{ or 1},$$
that is $\varphi$ is deterministic.\end{proof}
The following can be extracted from this proof:
\begin{corollary}
  A deterministic  $\operatorname{ev}_\sigma\in\mathbb{G}$ is invariant under wave function collapse.
\end{corollary}
The set of deterministic permutations, $G_{\mathbb{G}}$, therefore coincides with the set of characters $\Omega(C(\mathbb{G}))\subset \mathbb{G}$, and is a finite group whenever non-empty (Corollary \ref{grouplaw}). This implies that the set $G_\mathbb{G}$ coincides with the set $\widetilde{\mathbb{G}}$ of Kalantar and Neufang. A \emph{random permutation in }$\mathbb{G}$ is a convex combination of deterministic permutations, and the convex hull of $G_{\mathbb{G}}$ is the set of random permutations in $\mathbb{G}$, which could also be denoted by $M_p(G_{\mathbb{G}})$.

\bigskip

The study of maximal classical subgroups of compact quantum groups has a long history. An equivalent approach, seen for example in Banica and Skalski  \cite{BS2}, is to quotient $C(\mathbb{G})$ by its commutator ideal. Formally these approaches can give empty sets: the character space is empty if and only if the commutator ideal is the whole algebra. Studying instead the universal version $C_u(\mathbb{G})$ at least guarantees a counit, so that $e\in G_{\mathbb{G}}$, and $G_{\mathbb{G}}$ is a group (Corollary \ref{grouplaw}).

\bigskip

With an algebra of  functions on a finite quantum permutation group $F(\mathbb{G})$, a \emph{truly quantum} permutation is any quantum permutation zero on all one dimensional factors. More generally, for $\mathbb{G}\subseteq S_N^+$. working with the enveloping von Neumann algebra $\ell^{\infty}(\mathbb{G})$, in which $\imath:C(\mathbb{G})\hookrightarrow \ell^{\infty}(\mathbb{G})$ embeds,  a deterministic permutation $\operatorname{ev}_\sigma\in \mathbb{G}$ extends to a normal state $\omega_\sigma$ on $\ell^{\infty}(\mathbb{G})$, and thus has a \emph{support projection} $p_\sigma\in \ell^{\infty}(\mathbb{G})$ which is the smallest projection $p_\sigma\in \ell^{\infty}(\mathbb{G})$ such that $\omega_\sigma(p_\sigma)=1$, so that for any projection $p\in\ell^{\infty}(\mathbb{G})$ such that $\omega_\sigma(p)=1$, $p_\sigma\leq p$. Note that
$$\operatorname{ev}_{\sigma}(u_{\sigma(j)j})=\omega_\sigma(\imath(u_{\sigma(j)j}))=1\implies p_{\sigma}\leq \imath(u_{\sigma(j)j})\implies p_{\sigma}=\imath(u_{\sigma(j)j})p_{\sigma}=p_{\sigma}\imath(u_{\sigma(j)j}).$$
Any pair $\sigma\neq\tau$ of permutations are distinguished by some $\sigma(j)\neq \tau(j)$,
$$p_{\sigma}p_{\tau}=p_{\sigma}\imath(u_{\sigma(j)j})\imath(u_{\tau(j)j})p_{\tau}=p_{\sigma}\imath(u_{\sigma(j)j}u_{\tau(j)j})p_{\tau}=0.$$
Define:
\begin{equation} p_C=\sum_{\sigma\in G_{\mathbb{G}}}p_\sigma,\label{classsupp}\end{equation}
and define a quantum permutation $\varphi\in \mathbb{G}$ as \emph{truly quantum} if its normal extension $\omega_\varphi\in \mathcal{S}(\ell^{\infty}(\mathbb{G}))$ has the property that $\omega_\varphi(p_C)=0$. If $G_{\mathbb{G}}$ is empty, $p_C=0$ and all quantum permutations in $\mathbb{G}$ are truly quantum.

\bigskip

A quick consideration shows that if an algebra of functions on a quantum permutation group is a direct sum with a one-dimensional factor $\mathbb{C}f_i$, then the state $f^i:f_i\mapsto 1$ is deterministic.
\begin{example} \label{KP}
The \emph{Kac--Paljutkin quantum group of order eight} \cite{KP6}, $\mathfrak{G}_0$, has algebra of functions structure:
\begin{equation}F(\mathfrak{G}_0)=\mathbb{C}f_1\oplus\mathbb{C}f_2\oplus\mathbb{C}f_3\oplus\mathbb{C}f_4\oplus M_2(\mathbb{C}).\label{KP1}\end{equation}
Where $I_2\in M_2(\mathbb{C})$ the identity, and the projection
$$p:=\left(0,0,0,0,\left(\begin{array}{cc}
                    \frac12 & \frac12 e^{-i\pi/4} \\
                    \frac12 e^{+i\pi/4} & \frac12
                  \end{array}\right)\right),$$
a concrete exhibition of $\mathfrak{G}_0\subset S_4^+$ (Th. 1.1 (8), \cite{4Pts}) comes by the magic unitary:
\begin{align}
u:=\left[\begin{array}{cccc}
                            f_1+f_2 & f_3+f_4 & p & I_2-p \\
                            f_3+f_4 & f_1+f_2 & I_2-p & p \\
                            p^T & I_2-p^T & f_1+f_3 & f_2+f_4 \\
                            I_2-p^T & p^T & f_2+f_4 & f_1+f_3
                          \end{array}\right].\label{KPMU}
\end{align}
The one dimensional factors give deterministic permutations, $f^1=\operatorname{ev}_e$, $f^2=\operatorname{ev}_{(34)}$, $f^3=\operatorname{ev}_{(12)}$ and $f^4=\operatorname{ev}_{(12)(34)}$, so that
$G_{\mathfrak{G}_0}\cong\mathbb{Z}_2\times \mathbb{Z}_2$. Given a quantum permutation $\varphi\in \mathfrak{G}_0$, measurement with an $x(j)$ will see collapse to either a random permutation or a state on the $M_2(\mathbb{C})$ factor:  a  \emph{truly quantum} permutation. This can be illustrated using the notion of a \emph{quantum automorphism group of a finite graph} \cite{ba3}.

\bigskip

A finite graph $X$ with $N$ vertices has a quantum automorphism group $G^+(X)\subseteq S_N^+$. The quantum automorphism group of the below graph is the quantum hyperoctahedral group $H_2^+\subset S_4^+$ \cite{bbc2}:

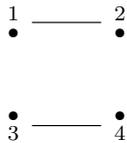
$$\begin{tikzcd}
\substack{1\\\bullet} \arrow[r,no head] & \substack{2\\\bullet} \\
\substack{\bullet\\ 3} \arrow[r,no head] & \substack{\bullet\\ 4}
\end{tikzcd}$$

The Kac--Paljutkin quantum group $\mathfrak{G}_0\subset S_4^+$ is a quantum subgroup $\mathfrak{G}_0\subset H_2^+\subset S_4^+$, and thus  $\varphi\in \mathfrak{G}_0$ may be viewed as a quantum automorphism of $X$, denoted here by $\varphi(X)$:
$$\begin{tikzcd}
\substack{\varphi(1)\\\bullet} \arrow[r,no head] & \substack{\varphi(2)\\\bullet} \\
\substack{\bullet\\ \varphi(3)} \arrow[r,no head] & \substack{\bullet\\ \varphi(4)}
\end{tikzcd}$$
Note $\varphi(1)$ and $\varphi(2)$ are entangled in the sense that if measurement of $x(1)$ yields e.g. $\varphi(1)=2$, then subsequent measurement of $x(2)$ yields $\varphi(2)=1$ with probability one:
$$\begin{tikzcd}
\substack{2\\\bullet} \arrow[r,no head] & \substack{1\\\bullet} \\
\substack{\bullet\\ \widetilde{u_{21}}\varphi(3)} \arrow[r,no head] & \substack{\bullet\\ \widetilde{u_{21}}\varphi(4)}
\end{tikzcd}$$
 Going back to $\varphi(X)$, if measurement of $x(1)$ yields $\varphi(1)=1$ or $\varphi(1)=2$, then $\varphi$ collapses to a random automorphism, with no non-classical behaviour observable. However, if measurement of $x(1)$ yields $\varphi(1)=3$ or $\varphi(1)=4$ then $\varphi$ collapses to a \emph{truly quantum} automorphism, and moreover there is an uncertainty principle about $x(3)$ and $x(4)$ in the sense that, e.g.
 $$\Phi(\widetilde{u_{31}}\varphi)=\begin{bmatrix}
                                           0 & 0 & \frac12 & \frac12 \\
                                           0 & 0 & \frac12 & \frac12 \\
                                           1 & 0 & 0 & 0 \\
                                           0 & 1 & 0 & 0
                                         \end{bmatrix}$$
 That is, if measurement of $x(1)$ reveals that $\widetilde{u_{31}}\varphi(X)$ is:
 $$\begin{tikzcd}
\substack{3\\\bullet} \arrow[r,no head] & \substack{4\\\bullet} \\
\substack{\bullet\\ \widetilde{u_{31}}\varphi(3)} \arrow[r,no head] & \substack{\bullet\\\widetilde{u_{31}}\varphi(4)}
\end{tikzcd},$$
it is the case that:
$$\mathbb{P}[\widetilde{u_{31}}\varphi(3)=1]=\frac12=\mathbb{P}[\widetilde{u_{31}}\varphi(3)=2].$$
This uncertainty principle will persist: if subsequent measurement shows that $\widetilde{u_{31}}\varphi(3)=1$, there will be more uncertainty:
$$\mathbb{P}[\widetilde{u_{13}}\widetilde{u_{31}}\varphi(1)=3]=\frac12=\mathbb{P}[\widetilde{u_{13}}\widetilde{u_{31}}\varphi(1)=4].$$
and non-classical behaviour can be exhibited, such as:
$$[\varphi(1)=4]\succ[\varphi(3)=1]\succ[\varphi(1)=3].$$

\end{example}

\section{Quantum group law and identity}\label{law}
\subsection{Quantum Group Law}
In the classical case of a finite group $G\subseteq S_N$, for $\sigma,\,\tau\in G$, the group law is encoded within the convolution of the pure states $\operatorname{ev}_{\sigma}$ and $\operatorname{ev}_{\tau}$:
$$(\operatorname{ev}_{\tau}\star \operatorname{ev}_{\sigma})\mathds{1}_{j\to i}=(\operatorname{ev}_{\tau}\otimes \operatorname{ev}_{\sigma})\Delta(\mathds{1}_{j\to i})=(\operatorname{ev}_{\tau}\otimes \operatorname{ev}_{\sigma})\sum_{k}(\mathds{1}_{k\to i}\otimes \mathds{1}_{j\to k})=\mathds{1}_{j\to i}(\tau\sigma).$$
The same game can be played with quantum permutations:
\begin{definition}
The \emph{quantum group law} $\mathbb{G}\times \mathbb{G}\to \mathbb{G}$ is the convolution, $\varphi_2\star\varphi_1:=(\varphi_2\otimes\varphi_1)\Delta$.
\end{definition}

\begin{proposition}

The Birkhoff slice is multiplicative:
$$\Phi(\varphi_2\star \varphi_1)=\Phi(\varphi_2)\Phi(\varphi_1),\qquad(\varphi_2,\,\varphi_1\in\mathbb{G}).$$
\end{proposition}

\begin{corollary}\label{grouplaw}\label{coamen}
The set of deterministic permutations $G_\mathbb{G}$ is either a group, or empty. It is a group if and only if  $\varepsilon\in\mathbb{G}$. Therefore if a quantum permutation group $\mathbb{G}$ is coamenable, or the algebra of continuous functions $C(\mathbb{G})\cong C_u(\mathbb{G})$, then $G_\mathbb{G}$  is a group. In particular, if  $\mathbb{G}\subseteq S_4^+$, then $G_\mathbb{G}$  is a group.
\end{corollary}If $G_{\mathbb{G}}$ is a group, the counit plays precisely the role of the identity:
$$\varphi\star\varepsilon=\varphi=\varepsilon\star \varphi\qquad (\varphi\in\mathbb{G}).$$ Restricted to $G_{\mathbb{G}}$, precomposing $\operatorname{ev}_\sigma$ with the antipode $S:C(\mathbb{G})\to C(\mathbb{G})$ gives an inverse:
$$\operatorname{ev}_\sigma\circ S=\operatorname{ev}_{\sigma^{-1}}\implies \operatorname{ev}_{\sigma^{-1}}\star \operatorname{ev}_{\sigma}= \operatorname{ev}_{\sigma}\star \operatorname{ev}_{\sigma^{-1}}=\operatorname{ev}_e=\varepsilon.$$
For a general quantum permutation, it might be more accurate to call $\varphi^{-1}:=\varphi\circ S$ the \emph{reverse} of $\varphi$ in the sense that
\begin{align}
\varphi^{-1}(|u_{i_nj_n}\cdots u_{i_1j_1}|^2)&=\varphi(|u_{j_ni_n}\cdots u_{j_1i_1}|^2)\label{reverse}
\\ \implies \mathbb{P}[[\varphi^{-1}(i_n)=j_n]\succ\cdots\succ [\varphi^{-1}(i_1)=j_1]]&=\mathbb{P}[[\varphi(j_n)=i_n]\succ \cdots\succ [\varphi(j_n)=i_n]]\nonumber
\end{align}

\bigskip

Remarkably for a piece about compact quantum groups, the \emph{Haar state} has not yet been introduced. The following is equivalent to more conventional definitions.
\begin{definition}
A quantum permutation group $\mathbb{G}$ has a quantum permutation $h_{\mathbb{G}}$ called the \emph{Haar state} that is the unique annihilator for the quantum group law, that is for all $\varphi\in\mathbb{G}$
$$h_{\mathbb{G}}\star \varphi=h_{\mathbb{G}}=\varphi\star h_{\mathbb{G}}.$$
\end{definition}
The non-zero elements of the Birkhoff slice $\Phi(h_{\mathbb{G}})$ are   equal along rows and columns. The Haar state can be thought of as the ``maximally random'' quantum permutation: in the classical case of $S_N$ it corresponds to the uniform measure on $S_N$.

\subsection{`Abelian'  Quantum Permutation Groups}\label{duals}
Given a compact group $G$, the algebra of continuous functions analogue of ``\emph{$G$ is abelian}'' is that ``$C(G)$ \emph{is cocommutative}'', that is to say $\Delta=\tau\circ\Delta$, where $\tau(f\otimes g)=g\otimes f$ is the flip map. In this sense an abelian compact quantum group is given by a cocommutative algebra of continuous functions $C(\widehat{\Gamma})$, that is an algebra of continuous functions on the dual of a discrete group $\Gamma$. As the quantum permutations in $\widehat{\Gamma}$ are in the state space of $C(\widehat{\Gamma})$, which is some class of positive definite functions on $\Gamma$ with pointwise, commutative multiplication, this idea that duals of discrete groups are abelian is trivial through Definition \ref{defqp1}.

\bigskip

Consider a cyclic group $\langle \gamma\rangle$ generated by an element $\gamma$ of order $N$. For $\omega:=\exp(2\pi i/N)$, consider the following vector in $F(\widehat{\langle\gamma\rangle})^N$:
\begin{equation}\frac{1}{N}\left[\begin{array}{c}
                                     e+\gamma+\gamma^2+\cdots +\gamma^{N-1} \\
                                     e+\omega \gamma+\omega^2 \gamma^2+\cdots +\omega^{N-1} \gamma^{N-1} \\
                                     e+\omega^2 \gamma+(\omega^2)^2 \gamma^2+\cdots +(\omega^2)^{N-1}\gamma^{N-1} \\
                                     \cdots \\
                                     e+\omega^{N-1}\gamma+(\omega^{N-1})^2 \gamma^2 +\cdots +(\omega^{N-1})^{N-1} \gamma^{N-1}
                                   \end{array}\right].\label{circ}\end{equation}
A circulant matrix defined by this vector is a magic unitary for $\widehat{\langle\gamma\rangle}\subseteq S_N^+$:
$$[u^{\gamma}]_{i,j}:=\frac{1}{N}\sum_{\ell=1}^N\omega^{(i-j)\ell} \gamma^{\ell}.$$
Note that
$$\gamma=u_{11}^{\gamma}+\omega^{N-1}u_{21}^{\gamma}+\omega^{N-2}u_{31}^{\gamma}+\cdots +\omega u_{N,1}^{\gamma}.$$
 Let $\Gamma=\langle\gamma_1,\dots\gamma_k\rangle$ be a finitely generated discrete group, with generators of finite order $N_1,N_2,\dots,N_k$. Then the dual $\widehat{\Gamma}\subseteq S_N^+$ where $N:=\sum_pN_p$ via the block magic unitary:

$$u=\left(\begin{array}{cccc}
                               u^{\gamma_1} & 0 & \cdots & 0 \\
                               0 & u^{\gamma_2} & \cdots & 0 \\
                               \vdots & \vdots & \ddots & 0 \\
                               0 & 0 & \cdots & u^{\gamma_k}
                             \end{array}\right).$$

                             Note that due to the fact that the $u^{\gamma_p}$ are circulant matrices, the entries of each $u^{\gamma_p}$ commute. Such a construction was used to prove Theorem \ref{coamenable}. It is long known that duals of finite groups $G$ are quantum permutation groups, but earlier references such as \cite{BBN} placed $\widehat{G}\subset S_{|G|^2}^+$. That result can be shown by considering each group element an order $|G|$ generator, so making for each $g_p$ a $|G|\times |G|$ magic unitary $u^{g_p}$, and forming the block matrix $\operatorname{diag}(u^{g_1},\dots,u^{g_{|G|}})$.
Smaller embeddings of duals of finite groups abound: an induction on $|G|$ shows that $\widehat{G}\subseteq S_{|G|}^+$. 

\bigskip

Let $G$ be a finite group. Where $\operatorname{Irr}(G)$ is an index set for a maximal set of pairwise inequivalent \emph{unitary} irreducible representations $\rho_\alpha:G\to M_{d_\alpha}(\mathbb{C})$, the algebra of functions on the dual has algebra
$$F(\widehat{G})=\bigoplus_{\alpha\in \operatorname{Irr}(G)}M_{d_\alpha}(\mathbb{C}).$$
When looking at concrete examples, sometimes it is easier to look at the regular representation:
$$\pi(F(\widehat{G}))\subset B(\mathbb{C}^{|G|}),\quad g:e_h\mapsto e_{gh}.$$
Each one dimensional representation gives a deterministic permutation. The quantum permutations are positive definite functions on $G$. The function $\mathds{1}_G\in \widehat{G}$ is the counit $F(\widehat{G})\to \mathbb{C}$.  That $\widehat{G}$ is abelian implies that the group $G_{\widehat{G}}$ of deterministic permutations is abelian. If $G$ is a simple group, either there are no truly quantum permutations, and $G\cong \mathbb{Z}_p$ for a prime $p$, and $\widehat{G}=G$; or $G_{\widehat{G}}$ is the trivial group. The dual of the symmetric group for $N\geq 2$ has only two deterministic permutations: one is the counit $\varepsilon=\mathds{1}_G$, and the other is the sign representation, an order two deterministic permutation: $\sum_{\sigma\in S_N}\operatorname{sgn}(\sigma)\delta_\sigma$.   The dual of the quaternion group has four deterministic permutations $G_{\widehat{Q}}\cong \mathbb{Z}_2\times\mathbb{Z}_2$,  the dual of a dihedral group has either two or four depending on the order. A \emph{perfect} (finite) group has only one representation of degree one, and the smallest perfect group is $A_5$. In this picture,  Pontryagin duality for a finite abelian group $G$ is nothing but $\widehat{G}$ having no truly quantum permutations: all the representations are one dimensional, and hence deterministic.

\bigskip

Suppose that $\Gamma=\langle \gamma_1,\dots,\gamma_k\rangle$ is a discrete group such that $\widehat{\Gamma}\subseteq S_N^+$. Partition the symbols $1,\dots,N$ into blocks $B_1,\dots,B_k$ with $B_p$ of size $N_p$. The fact that the blocks of $\operatorname{diag}(u^{\gamma_1},\dots,u^{\gamma_k})$ are circulant matrices implies that if for $j\in B_p$, the measurement of a quantum permutation $\varphi\in \widehat{\Gamma}$ with $x(j)$ will see wave function collapse such that the restriction of the state to $B_p$ is now deterministic in the sense that if, for $r,s\in B_p$, $\varphi (u_{rs})\neq 0$, then the matrix $[\Phi(\widetilde{u_{rs}}(\varphi))]_{i,j\in B_p}$ is a permutation matrix. For example, for $N\geq 9$, and generators $\sigma$ of order two, and $\tau$ of order three, consider $\widehat{S_N}\subset S_5^+$ with blocks $B_1=\{1,2\}$ from $u^{\sigma}$ and $B_2=\{3,4,5\}$ from $u^{\tau}$. Let $\varphi\in \widehat{S_N}$ be a quantum permutation. Measure with
$$x(4)=3u_{34}+4u_{44}+5u_{54}.$$
Suppose that the measurement yields $x(4)=5$. Then the quantum permutation collapses to:
$$\widetilde{u_{54}}(\varphi)=\frac{\varphi(u_{54}\cdot u_{54})}{\varphi(u_{54})},$$
and the circulant nature of $u^{\tau}$ implies that the Birkhoff slice
$$\Phi(\widetilde{u_{54}}\varphi)=\left(\begin{array}{cc|ccc}
                                                                           \alpha & 1-\alpha & 0 & 0 & 0 \\
                                                                           1-\alpha & \alpha & 0 & 0 & 0 \\ \hline
                                                                           0 & 0 & 0 & 0 & 1 \\
                                                                           0 & 0 & 1 & 0 & 0 \\
                                                                           0 & 0 & 0 & 1 & 0
                                                                         \end{array}\right),$$
                                                                         that is the measurement of $x(4)=5$ collapses the quantum permutation in such a way that it is deterministic on $B_2$.
Unlike in the case of $S_3^+$, while $x(1)$ and $x(2)$ are now entangled, the measurement of $x(1)$ can disturb the state in such a way that the complete certainty around $B_2$ is disturbed. The same phenomenon necessarily occurs for the dual of any non-abelian finite group. The circulant nature of the blocks implies that for all $r,s\in B_p$, the measurement of $x(r)$ determines $x(s)$, and all such measurements could be denoted $x(B_p)$. The dual of a discrete group $\Gamma=\langle\gamma_1,\dots,\gamma_k\rangle $ could model a $k$ particle ``entangled'' quantum system, where the $p$-th particle, corresponding to the block $B_p$, has $N_p$ states, labelled $1,\dots,N_p$. Full information about the state of all particles is in general impossible, but measurement with $x(B_p)$ will see collapse of the $p$th particle to a definite state. Only the deterministic/random permutations in $\widehat{\Gamma}$ would correspond to classical states.

\section{Phenomena}\label{randomwalks}

\subsection{Quasi-subgroups}\label{quasi}
If $G$ is a classical finite group, subsets $S\subseteq G$ that are closed under the group law are subgroups. More generally, for the algebra of continuous functions $C(G)$ on a classical compact group,  states $C(G)\to \mathbb{C}$ correspond via integration to Borel probability measures in $M_p(G)$, and in the context of this work could be called \emph{random permutations} (of an infinite set if $G$ is infinite). In this context the convolution might be called the random group law, and the Kawada--It\^{o} theorem says that random permutations idempotent with respect to the random group law are Haar measures/states of compact subgroups of $G$ \cite{Kaw}.

\bigskip

 Suppose that $\mathbb{H}\subseteq \mathbb{G}\subseteq S_N^+$ by $\pi:C(\mathbb{G})\to C(\mathbb{H})$. The Haar state of $\mathbb{H}$ in $\mathbb{G}$, $h_{\mathbb{H}}\circ \pi$, is an idempotent state in $\mathbb{G}$, called a \emph{Haar idempotent}. Compact quantum groups can have non-Haar idempotents. In our context these are quantum permutations $\phi\in\mathbb{G}$ idempotent with respect to the quantum group law, $\phi\star \phi=\phi$, that are not the Haar state on any compact subgroup. In the Gelfand picture, idempotent states correspond to measures uniform on virtual objects called \emph{quasi-subgroups}\footnote{this terminology is from \cite{Kasp} and has nothing to do with cancellative magmas}.

\bigskip

In the case of finite quantum groups, an idempotent state $\phi_{\mathbb{S}}$ is associated to a \emph{group-like projection} $\mathds{1}_{\mathbb{S}}$, which is also the support projection of $\phi_{\mathbb{S}}$ (see \cite{ERG} for more including original references), and therefore it is tenable to define:
\begin{definition}
Let $\mathbb{G}$ be a finite quantum permutation group with an idempotent state $\phi_{\mathbb{S}}\in\mathbb{G}$ and associated group-like projection $\mathds{1}_{\mathbb{S}}\in F(\mathbb{G})$. The associated \emph{quasi-subgroup} $\mathbb{S}\subseteq \mathbb{G}$ is given by:
\begin{equation}\mathbb{S}:=\left\{\varphi\in\mathbb{G}:\varphi(\mathds{1}_{\mathbb{S}})=1\right\}.\label{null}\end{equation}
\end{definition}
Quasi-subgroups of finite quantum groups behave very much like quantum subgroups: they are closed under the quantum group law (Prop. 3.12, \cite{ERG}), they contain the identity $\varepsilon$, and they are closed under precomposition with the antipode (\cite{Land}).

\bigskip

Let $\Gamma=\langle \gamma_1,\dots,\gamma_k\rangle$ be a discrete group with generators of finite order. The idempotent states  $\mathds{1}_{\Lambda}$ are given by subgroups $\Lambda\subseteq \Gamma$. They have support projections:
$$\chi_{\Lambda}=\frac{1}{|\Lambda|}\sum_{\lambda\in\Lambda}\lambda,$$
and so give quasi-subgroups $\mathbb{S}_{\Lambda}\subseteq \widehat{\Gamma}$ :
$$\mathbb{S}_{\Lambda}:=\{\varphi\in \widehat{\Gamma}:\varphi(\lambda)=1\text{ for all }\lambda\in \Lambda\}.$$
They give quantum subgroups only when $\Lambda$ is a normal subgroup of $\Gamma$. This is an illustration of the fact that quotients $\Gamma \to \Gamma/\Lambda$ linearly extend to quotients $C(\widehat{\Gamma})\to C(\widehat{\Gamma/\Lambda})$, which give rise to quantum subgroups $\widehat{\Gamma/\Lambda}\subseteq\widehat{\Gamma}$ via $$\pi:C(\widehat{\Gamma})\to C(\widehat{\Gamma/\Lambda});\quad \sum_{\gamma\in \Gamma}\alpha_\gamma \gamma\mapsto \sum_{\gamma \in\Gamma}\alpha_\gamma [\gamma].$$
Pal's idempotents in the Kac--Paljutkin quantum group provide another example of non-Haar idempotents \cite{Pal}: where $f^1,\,f^4$ are dual to $f_1,\,f_4\in F(\mathfrak{G_0})$, and $E^{11},\,E^{22}$  dual to $E_{11},\,E_{22}$ in the $M_2(\mathbb{C})$ factor of $F(\mathfrak{G}_0)$, consider the convex hulls $\mathbb{S}_i:=\operatorname{co}(\{f^1,f^4,E^{ii}\})$, with associated idempotent states $\frac14 f^1+\frac14 f^4+\frac12 E^{ii}\in\mathfrak{G}_0$. These $\mathbb{S}_i$ are quasi-subgroups.

\bigskip

A very good question is: why are quasi-subgroups given by non-Haar idempotents not considered quantum subgroups? The conventional analysis for an idempotent state $\phi\in \mathbb{G}$ on a  finite quantum group $\mathbb{G}$ is to consider the ideal:
$$N_{\phi}:=\left\{g\in F(\mathbb{G}):\phi(g^*g)=0\right\}.$$
Franz and Skalski (Th. 4.5, \cite{idempotent}) show that $\phi$ is a Haar idempotent precisely when $N_\phi$ is two-sided, equivalently self-adjoint, equivalently $S$-invariant. Therefore when $\phi$ is a non-Haar idempotent, the quotient $F(\mathbb{G})/N_{\phi}$ doesn't have a quantum group structure.

\bigskip

That not all quasi-subgroups are quantum groups can be explained in terms of wave function collapse. For example for Pal's quasi-subgroup $\mathbb{S}_1:=\operatorname{co}(\{f^1,f^4,E^{11}\})$, the support of the idempotent state
$$\phi_1=\frac14 f^1+\frac14 f^4+\frac12 E^{11},$$
is $\mathds{1}_{\mathbb{S}_1}:=f_1+f_2+E_{11}$. Where $\mathfrak{G}_0\subset S_4$ is given by (\ref{KPMU}), consider the matrix unit $\varphi:=E^{11}$ in $\mathbb{S}_1$:
$$\mathbb{P}[\varphi(3)=1]=E^{11}(u_{13})=\frac{1}{2}.$$
Therefore the quantum permutation conditioned on $\varphi(3)=1$ is:
$$\widetilde{u_{13}}\varphi=\frac{\varphi(u_{13}\cdot u_{13})}{\varphi(u_{13})}=2E^{11}(u_{13}\cdot u_{13}).$$
But:
$$\widetilde{u_{13}}\varphi(\mathds{1}_{\mathbb{S}_1})=2E^{11}(u_{13}\mathds{1}_{\mathbb{S}_1} u_{13})=\frac12,$$
and this implies with (\ref{null}) that measurement with $u_{13}$ has conditioned the quantum permutation outside the quasi-subgroup.

\bigskip

So  quasi-subgroups that are not quantum subgroups are just like quantum subgroups: until you start measuring them and it is seen that they are not stable under wave function collapse. It can be shown that for $\mathrm{C}^*$-algebras generated by projections such as algebras of functions $F(\mathbb{G})$ on finite quantum permutation groups, if $p\in F(\mathbb{G})$ is a projection and $J:=F(\mathbb{G})p$ is a left but not both-sided ideal, then there exists a projection $q\in F(\mathbb{G})$ such that $qpq\not\in J$. That implies that for every quasi-subgroup $\mathbb{S}\subset \mathbb{G}$ that is not a quantum subgroup, there is a quantum permutation $\varphi\in \mathbb{S}$ and a projection $q\in F(\mathbb{G})$ such that $\widetilde{q}\varphi\not\in \mathbb{S}$.

\bigskip

This cannot happen with a finite quantum subgroup $\mathbb{H}\subseteq  \mathbb{G}$.  Let $\mathds{1}_{\mathds{H}}\in F(\mathbb{G})$ be the support projection of the Haar idempotent $h_{\mathbb{H}}\circ \pi$, and define $\mathbb{H}\subseteq \mathbb{G}$  by
$$\mathbb{H}:=\{\varphi\in\mathbb{G}\,:\,\varphi(\mathds{1}_{\mathbb{H}})=1\}.$$
For any $\varphi\in\mathbb{H}$ and projection $p\in F(\mathbb{G})$, the fact that $\mathds{1}_{\mathbb{H}}$ is central \cite{idempotent} implies that $\widetilde{p}\varphi\in\mathbb{H}$; that is quantum subgroups are invariant under wave function collapse.

\subsection{Fixed Points Phenomena\label{fpp} and Quantum Transpositions}
Given a quantum permutation group $\mathbb{G}\subseteq S_N^+$, define the number of fixed points observable:
$$\operatorname{fix}:=\sum_{j=1}^Nu_{jj}.$$
In general, the spectrum of $\operatorname{fix}$ contains non-integers: indeed by looking at a faithful representation $\pi(C(\widehat{D_{\infty}}))\subset B(L^2([0,1],M_2(\mathbb{C})))$ \cite{Raeburn}, it can be seen that for $\widehat{D_{\infty}}\subset S_4^+$, $\sigma(\operatorname{fix})=[0,4]$. For $\operatorname{fix}$ in universal $C(S_N^+)$, $\sigma(\operatorname{fix})=[0,N]$. Back with arbitrary $\mathbb{G}\subseteq S_N^+$, when the spectrum of $\operatorname{fix}$ is finite, such as in the case of a finite quantum permutation group, there is a spectral decomposition in $C(\mathbb{G})$:
\begin{equation}\operatorname{fix}=\sum_{\lambda \in\sigma(\operatorname{fix})}\lambda \,p_\lambda.\label{fixspect}\end{equation}

\begin{definition}
Where $\operatorname{fix}=\sum_{j=1}^Nu_{jj}$ has spectral decomposition (\ref{fixspect}), a quantum permutation in a finite quantum permutation group $\mathbb{G}\subseteq S_N^+$ \emph{has $\lambda$ fixed points} if $\varphi(p_\lambda)=1$. A quantum permutation in $\mathbb{G}\subseteq S_N^+$ with $N-2$ fixed points is a quantum \emph{transposition}.
\end{definition}
Note that if a quantum permutation has $\lambda$ fixed points, the trace of $\Phi(\varphi)$ is $\lambda$.

\bigskip

Define a magic unitary for $\widehat{S_3}\subset S_4^+$ by
\begin{align*}\left[\begin{array}{cc}
                           u^{(12)} & 0 \\ 0 & u^{(13)}
                          \end{array}\right].\end{align*}

The spectrum $\sigma(\operatorname{fix})=\{0,1,3,4\}$. The deterministic permutations $\operatorname{ev}_e$ (given by the trivial representation) and $\operatorname{ev}_{(12)(34)}$ (given by the sign representation) have four and zero fixed points. A quantum permutation with three fixed points is:
$$\varphi=\delta_e+\frac12\delta_{(12)}+\frac12\delta_{(13)}-\delta_{(23)}-\frac12\delta_{(123)}-\frac{1}{2}\delta_{(132)}.$$
It has Birkhoff slice:
$$\Phi(\varphi)=\left(\begin{array}{cccc}
                   3/4 & 1/4 & 0 & 0 \\
                   1/4 & 3/4 & 0 & 0 \\
                   0 & 0 & 3/4 & 1/4 \\
                   0 & 0 & 1/4 & 3/4
                 \end{array}\right).$$
Placing $\widehat{S_3}\subset S_8^+$ via $\operatorname{diag}(u^{\widehat{S_3}},u^{\widehat{S_3}})$ is one way to get a transposition, however note, reminding of $A_8\lhd S_8$, there is a periodicity to $\varphi\in S_8^+$:

$$\lim_{k\to \infty}\varphi^{\star2k}=\delta_e+\delta_{(23)}\text{; and }\lim_{k\to \infty}\varphi^{\star (2k+1)}= \delta_e-\delta_{(23)}.$$
 On the technical level, this is unlike the periodicity of the state uniform on permutations of odd parity because $\delta_e+\delta_{(23)}$ is not the Haar state on a quantum subgroup of $\widehat{S_3}$, so it doesn't make sense to say that $\langle (23)\rangle$ is normal in $S_3$. See \cite{ERG} for more.

\bigskip

\bigskip

Another quantum phenomenon is that there are quantum permutations with $N-1$ fixed points which are \emph{not} the identity. Consider the finite quantum group $\widehat{S_4}\subset S_5^+$ given by the magic unitary:
$$\left(\begin{array}{cc}u^{(12)} & 0 \\ 0 & u^{(234)}\end{array}\right)\in M_5(F(\widehat{S_4})).$$
Representing $\pi(F(\widehat{S_4}))\subset B(\mathbb{C}^{24})$ with the regular representation, and employing a CAS, it can be found that
$$\sigma(\operatorname{fix})=\left\{0,\frac{5-\sqrt{17}}{2},1,2,3,4,\frac{5+\sqrt{17}}{2},5\right\},$$
so that the phenomenon of quantum permutations with a non-integer number of fixed points occurs for $\widehat{S_4}\subset S_5^+$. Define subsets of $S_4$:
$$X_1:=\langle (34)\rangle,\,X_2:=(12)\langle(34)\rangle,\,X_3:=\{\sigma:\sigma(1)=1\}\backslash X_1,\,X_4:=\{\sigma:\sigma(2)=2\}\backslash X_1,$$ $$\,X_5:=\{(13)(24),(14)(23),(1423),(1324)\}\text{, and }X_6:=S_4\left\backslash \left(\bigcup_{\ell=1}^5 X_\ell\right)\right. .$$
Then the following quantum permutation has four fixed points and is not the identity/counit:
\begin{equation}\varphi:=\mathds{1}_{X_1}+\frac{1}{3}\mathds{1}_{X_2}+\frac{5}{6}\mathds{1}_{X_3}-\frac12\mathds{1}_{X_4}-\frac23 \mathds{1}_{X_5}-\frac16\mathds{1}_{X_6},\label{randtran}\end{equation}
and has Birkhoff slice:
$$\Phi(\varphi)=\left(\begin{array}{ccccc}
    2/3 & 1/3 & 0 & 0 & 0 \\
    1/3 & 2/3 & 0 & 0 & 0 \\
    0 & 0 & 8/9 & 1/18 & 1/18 \\
    0 & 0 & 1/18 & 8/9 & 1/18 \\
    0 & 0 & 1/18 & 1/18 & 8/9
  \end{array}\right)$$
As the convolution power in $\widehat{S_4}$ is pointwise multiplication, $\varphi^{\star k}\to \mathds{1}_{X_1},$ and there is convergence to a non-Haar idempotent.

\bigskip

For any integer $\ell\geq 2$, quantum permutation groups $\mathbb{G}\subseteq S_N^+$ are also quantum permutation groups $\mathbb{G}\subset S_{\ell N}^+$ via:
$$\operatorname{diag}(u,u,\dots,u)\in M_{\ell N}(C(\mathbb{G})),$$
and so if $\varphi\in \mathbb{G}$ has $N-2/\ell$ fixed points in $\mathbb{G}\subseteq S_N^+$, then $\varphi\in \mathbb{G}$ is a transposition in $\mathbb{G}\subseteq S_{\ell N}$, that is it has $\ell N-2$ fixed points. Therefore, $\varphi\in \widehat{S_4}$ given by (\ref{randtran}) is a quantum transposition in $\widehat{S_4}\subset S_{10}^+$ whose convolution powers do not exhibit periodicity.

\bigskip

When $\sigma(\operatorname{fix})$ is no longer finite, pass to the universal enveloping von Neumann algebra $\ell^{\infty}(\mathbb{G})$ of $C(\mathbb{G})$, which contains $\imath:C(\mathbb{G})\hookrightarrow \ell^{\infty}(\mathbb{G})$ and the spectral projections of elements of $C(\mathbb{G})$. Consider $\imath(\operatorname{fix})\in \ell^{\infty}(\mathbb{G})$, with spectral projections $\mathds{1}_S(\operatorname{fix})$, in particular $p_{\lambda}:=\mathds{1}_{\{\lambda\}}(\operatorname{fix})$. Where $\omega_{\varphi}$ is the normal extension of  $\varphi\in \mathbb{G}$, define:
 $$\mathbb{P}[\varphi\text{ has }\lambda\text{ fixed points}]:=\omega_{\varphi}(p_{\lambda}),$$
 and say that $\varphi$ has $\lambda$ fixed points if $\omega_{\varphi}(p_{\lambda})=1$.

 \bigskip

Let $\pi:C(S_N^+)\to C(\mathbb{G})$ give a quantum subgroup $\mathbb{G}\subset S_N^+$. Let $u^{\mathbb{G}}_{ij}$ be the generators of $C(\mathbb{G})$ and $u_{ij}$ the generators of (universal) $C(S_N^+)$. Suppose $\operatorname{fix}^{\mathbb{G}}=\sum_{j}u_{jj}^{\mathbb{G}}$ is of finite spectrum. Suppose $\varphi_0\in \mathbb{G}$ has $\lambda_0$ fixed points, and consider $\varphi:=\varphi_0\circ \pi$.  Consider the $\mathrm{C}^*$-algebras generated by $\operatorname{fix}=\sum_{j}u_{jj}$, $\mathrm{C}^*(\operatorname{fix})\cong C([0,N])$; and by $\operatorname{fix}^{\mathbb{G}}=\sum_{j}u_{jj}^{\mathbb{G}}$,  $\mathrm{C}^*(\operatorname{fix}^{\mathbb{G}})\cong C(\sigma(\operatorname{fix}^{\mathbb{G}}))$.  For $f\in \mathrm{C}^*(\operatorname{fix})$,
$$\pi(f)=\sum_{\lambda\in\sigma(\operatorname{fix}^{\mathbb{G}})}f(\lambda)\delta_\lambda.$$
By assumption, for  $\lambda_0\in\sigma(\operatorname{fix}^{\mathbb{G}})$, $\varphi_0(\delta_{\lambda_0})=1$ and so
$$\varphi(f)=\varphi_0\left(\sum_{\lambda\in\sigma(\operatorname{fix}^{\mathbb{G}})}f(\lambda)\delta_\lambda\right)=f({\lambda_0})\implies \varphi=\operatorname{ev}_{{\lambda_0}}.$$

The enveloping von Neumann algebra $\mathrm{C}^*(\operatorname{fix})^{**}\cong \ell^{\infty}([0,N])$,  $\ell^{\infty}([0,N])\subset B(\ell^2([0,N]))$, and $\varphi$ extends to $\omega_\varphi=\operatorname{ev}_{{\lambda_0}}$, a state on $\ell^{\infty}([0,N])$. The spectral projection $\mathbf{1}_{\{{\lambda_0}\}}(\operatorname{fix})$ is an element of $B(\ell^2([0,N]))$:

\begin{equation}f\mapsto f({\lambda_0})\delta_{{\lambda_0}},\label{fixed}\end{equation}

so that $\mathbf{1}_{\{{\lambda_0}\}}(\operatorname{fix})$ may be identified with $\delta_{{\lambda_0}}\in\ell^{\infty}([0,N])$ and indeed $\omega_\varphi(\mathbf{1}_{\{{\lambda_0}\}}(\operatorname{fix}))=\operatorname{ev}_{{\lambda_0}}(\delta_{{\lambda_0}})=1$, so that $\varphi$ also has ${\lambda_0}$ fixed points.

\bigskip

The quantum transposition $\varphi_{\text{tr}}$ studied by Freslon, Teyssier and Wang \cite{FTW} is a \emph{central state}, and it is the only central quantum transposition in $S_N^+$. \emph{Central states} such as $\varphi_{\text{tr}}$ have some nice properties:  for any irreducible representation $n\in\mathbb{N}_{\geq 0}$, there exists a scalar $\varphi_{\text{tr}}(n)$ such that $\varphi_{\text{tr}}(\rho^{(n)}_{ij})=\varphi_{\text{tr}}(n)\delta_{i,j}$, and as the matrix elements of the irreducible representations form a  basis of $C(S_N^+)$, they are completely determined by their restriction to the central algebra $C(S_N^+)_0$ generated by the characters as:
$$\varphi_{\text{tr}}(\chi_n)=\sum_{i=1}^{d_n}\varphi_{\text{tr}}(\rho^{(n)}_{ii})=d_n\varphi_{\text{tr}}(n).$$
The central algebra is commutative, and it follows from spectral theory that:
$$C(S_N^+)_0\cong C([0,N]).$$
The isomorphism from the characters to $C([0,N])$ is given by $\chi_n\mapsto (t\mapsto U_{2n}(\sqrt{t}/2))$, where $U_n$ are the Chebyshev polynomials of the second kind, and therefore, restricted to the central algebra, $\chi_0+\chi_1=\operatorname{fix}$. The state $\varphi_{\text{tr}}$ is given by $\operatorname{ev}_{N-2}\in C([0,N])^*$.  The normal extension of $\varphi_{\text{tr}}$ is also $\operatorname{ev}_{N-2}$, and indeed $\operatorname{ev}_{N-2}(p_{N-2})=1$, and because of (\ref{fixed}), $\operatorname{ev}_{N-2}$ is the unique central quantum transposition in $S_N^+$.

\subsection{Maximality of $S_N\subseteq S_N^+$}
As mentioned in Section \ref{subqpg}, there is the following maximality result:
\begin{theorem}
For $N\leq 5$, there is no intermediate $S_N\subset \mathbb{G}\subset S_N^+$.
\end{theorem}
The result is conjectured to be true for $N>5$ also.  A strong way to interpret the conjecture is to say that all that has to be added to $S_N$ to get the whole of $S_N^+$ is an arbitrary quantum permutation from $\mathbb{G}\subset S_N^+$. In this section the $u_{ij}\in C(S_N^+)$, and, without making it notationally explicit, all quantum permutations are assumed elements of $S_N^+$ via, for $\varphi_0'\in \mathbb{G}$ and $\operatorname{ev}_{\sigma}\in S_N$:
$$\varphi_0(u_{ij})=\varphi_0'\circ \pi_{\mathbb{G}}(u_{ij})=\varphi_0'(u_{ij}^{\mathbb{G}})\text{, and }\operatorname{ev}_{\sigma}(u_{ij})=\operatorname{ev}_{\sigma}'\circ \pi_{S_N}(u_{ij})=\mathds{1}_{j\to i}(\sigma).$$
Here $\pi_{\mathbb{G}}$ is the quotient $C(S_N^+)\to C(\mathbb{G})$. It is thus possible to convolve quantum permutations in $\mathbb{G}\subseteq S_N^+$ with deterministic permutations in $S_N$.

\bigskip

Let $\mathbb{G}_0\subset S_N^+$ and $\varphi_0\in \mathbb{G}_0$ any (non-classical) quantum permutation. Working with $C_u(\mathbb{G}_0)$, and $h_{S_N}$ the Haar state on $F(S_N)$ define:
$$\varphi':=\frac12\varepsilon+\frac12\varphi_0,\text{ and }\varphi:=h_{S_N}\star \varphi'.$$
The Ces\'{a}ro averages
$$\frac{1}{n}\sum_{k=1}^n \varphi^{\star k}\overset{w^{*}}{\longrightarrow}\phi_{\varphi_0},$$
an idempotent state in $S_N^+$.  There are three possibilities here.
\begin{enumerate}
  \item $\phi_{\varphi_0}=h_{S_N^+}$, the Haar state on $S_N^+$;
  \item $\phi_{\varphi_0}=h_{\mathbb{G}}$, for an intermediate quantum group $S_N\subset \mathbb{G}\subset S_N^+$;
  \item $\phi_{\varphi_0}$ is a non-Haar idempotent, giving an intermediate quasi-subgroup $S_N\subset \mathbb{S}\subset S_N^+$.
\end{enumerate}
The conjecture of maximality $S_N\subseteq S_N^+$ believes that (2) cannot happen but does not preclude (3). If neither (2) nor (3) can happen, then \emph{any} non-classical quantum permutation in \emph{any} quantum permutation group together with $h_{S_N}$ generates the whole of $S_N^+$.

\bigskip

This following illustrates how convolving deterministic permutations with non-classical quantum permutations can `move' the quantumness around.

\begin{proposition}\label{alpha1}
Suppose $\mathbb{G}\subseteq S_N^+$. Then for all states $\varphi\in \mathbb{G}$ and $\sigma,\tau\in S_N$:
$$(\operatorname{ev}_{\sigma^{-1}}\star \varphi\star \operatorname{ev}_\tau)(|u_{i_nj_n}\cdots u_{i_1j_1}|^2)=\varphi(|u_{\sigma(i_n)\tau(j_n)}\cdots u_{\sigma(i_1)\tau(j_1)}|^2)$$
\end{proposition}
\begin{proof}
Consider $(\operatorname{ev}_{\sigma^{-1}}\star\varphi\star\operatorname{ev}_{\tau})(|u_{i_nj_n}\cdots u_{i_1j_1}|^2)$
\begin{align*}
    &=(\operatorname{ev}_{\sigma^{-1}}\otimes \varphi\otimes \operatorname{ev}_{\tau})\Delta^{(2)}(u_{\sigma(i_1)\tau(j_1)}\cdots u_{\sigma(i_n)\tau(j_n)}\cdots u_{\sigma(i_1)\tau(j_1)}) \\
   & =(\operatorname{ev}_{\sigma^{-1}}\otimes \varphi\otimes \operatorname{ev}_{\tau})\Delta^{(2)}(u_{i_1j_1})\cdots \Delta^{(2)}(u_{i_nj_n})\cdots \Delta^{(2)}(u_{i_1j_1}) \\
   &=(\operatorname{ev}_{\sigma^{-1}}\otimes \varphi\otimes \operatorname{ev}_{\tau})\left(\sum_{k_1,k_2=1}^Nu_{i_1k_1}\otimes u_{k_1k_2}\otimes u_{k_2j_1}\right)\times
    \\ &\cdots\left(\sum_{k_{2n-1},k_{2n}=1}^Nu_{i_nk_{2n-1}}\otimes u_{k_{2n-1}k_{2n}}\otimes u_{k_{2n}j_n}\right)\times \\ & \cdots\left(\sum_{k_{4n-3},k_{4n-2}=1}^Nu_{i_1k_{4n-3}}\otimes u_{k_{4n-3}k_{4n-2}}\otimes u_{k_{4n-2}j_1}\right) \\&
    =\sum_{k_1,\dots,k_{4n-2}=1}^N\operatorname{ev}_{\sigma^{-1}}(u_{i_1k_1}\cdots u_{i_nk_{2n-1}}\cdots u_{i_1k_{4n-3}})\times \\&\varphi(u_{k_1k_2}\cdots u_{k_{2n-1}k_{2n}}\cdots u_{k_{4n-3}k_{4n-2}})     \operatorname{ev}_{\tau}(u_{k_2j_1}\cdots u_{k_{2n}j_n}\cdots u_{k_{4n-2}j_1})
\end{align*}
The evaluation states are characters and moreover $\operatorname{ev}_{\sigma^{-1}}(u_{ij})=1$ if and only if $\sigma(i)=j$, and this implies that non-zero terms come from $k_1=\sigma(i_1),\dots, k_{2n-1}=\sigma(i_n),\dots,k_{4n-3}=\sigma(i_1)$ and also $k_2=\tau(j_1),\dots,k_{2n}=\tau(j_n),\dots,k_{4n-2}=\tau(j_1)$ and the result follows.
\end{proof}

\begin{corollary}\label{alpha}
Let $\varphi_0$ be a quantum permutation in $\mathbb{G}\subseteq S_{N}^+$. Then, where $\varphi:=\operatorname{ev}_{\sigma^{-1}}\star \varphi_0\star\operatorname{ev}_{\tau}$:
$$\mathbb{P}[[\varphi(j_n)=i_n]\succ \cdots \succ [\varphi(j_1)=i_1]]=\mathbb{P}[[\varphi_0(\sigma(j_n))=\tau(i_n)]\succ \cdots \succ [\varphi_0(\sigma(j_1))=\tau(i_1)]].$$

\end{corollary}

\subsection{Orbits and Orbitals}
One barrier to attacks on the maximality conjecture via methods inspired by the previous section is the non-transitivity of higher \emph{orbitals}.  The study of orbits (or one-orbitals) was initiated by a number of authors. Lupini, Man\v{c}inska, and Roberson \cite{lupini}; and (independently)  Banica and Freslon \cite{Modelling}  introduced orbits around the same time, but both missed the earlier study of Huang \cite{Hua}. The study of orbitals (or two-orbitals) was initiated by Lupini, Man\v{c}inska, and Roberson \cite{lupini}. In this section orbits, orbitals, and three-orbitals are studied, in the language of quantum permutations, and a new (conjectured) counter-example to the  transitivity of the three-orbital relation is given.

\bigskip\begin{definition}
Let $\mathbb{G}\subseteq S_N^+$. Define a relation $\sim_k$ on $\{1,2,\dots,N\}^k$ by
$$(i_k,\dots,i_1)\sim(j_k,\dots,j_1)\iff u_{i_kj_k}\cdots u_{i_1j_1}\neq 0.$$
We call $\sim_1$ the orbit relation, $\sim_2$ the orbital relation, and $\sim_k$ the $k$-orbital relation.
\end{definition}
Both $\sim_1$ and $\sim_2$ are equivalence relations, their equivalence classes called orbits and orbitals \cite{lupini}. The non-trivial part of this business is to demonstrate the transitivity of the orbital relation.  To illustrate, using the Gelfand--Birkhoff picture, consider $\mathbb{G}\subseteq S_4^+$. Suppose that $u_{13}u_{32}\neq 0$ and $u_{32}u_{24}\neq 0$ so that there exists $\rho_0,\varphi_0\in\mathbb{G}$ such that:
$$\mathbb{P}[[\rho_0(3)=1]\succ[\rho_0(2)=3]]>0\text{, and }\mathbb{P}[[\varphi_0(2)=3]\succ[\varphi_0(4)=2]]>0.$$
 Consider $\rho:=\widetilde{u_{32}}\rho_0$ and $\varphi:=\widetilde{u_{23}}\varphi_0$, the quantum permutations $\rho_0$ and $\varphi_0$ conditioned by $\rho_0(2)=3$ and $\varphi_0(4)=2$, respectively:
$$\rho(\cdot )=\frac{\rho_0(u_{32}\cdot u_{32})}{\rho_0(u_{32})}\text{ and }\varphi(\cdot )=\frac{\varphi_0(u_{24}\cdot u_{24})}{\varphi_0(u_{24})}.$$

Below left there is a schematic of the quantum permutation $\rho_0$, and below right a schematic for the quantum permutation $\rho$, which maps $\rho(2)=3$ with probability one:
$$
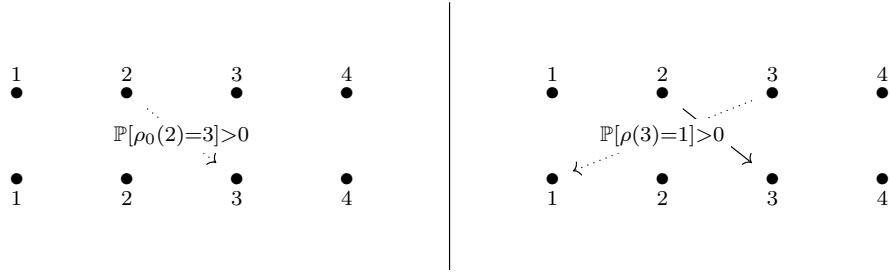
\begin{tikzcd}
                      &                                                                                                   &                       &                       & {} \arrow[ddd, no head] &                       &                                 &                                                                                   &                       \\
\overset{1}{\bullet}  & \overset{2}{\bullet} \arrow[rd, "{\mathbb{P}[\rho_0(2)=3]>0}" description, dotted] & \overset{3}{\bullet}  & \overset{4}{\bullet}  &                         & \overset{1}{\bullet}  & \overset{2}{\bullet} \arrow[rd] & \overset{3}{\bullet} \arrow[lld, "{\mathbb{P}[\rho(3)=1]>0}" description, dotted] & \overset{4}{\bullet}  \\
\underset{1}{\bullet} & \underset{2}{\bullet}                                                                             & \underset{3}{\bullet} & \underset{4}{\bullet} &                         & \underset{1}{\bullet} & \underset{2}{\bullet}           & \underset{3}{\bullet}                                                             & \underset{4}{\bullet} \\
                      &                                                                                                   &                       &                       & {}                      &                       &                                 &                                                                                   &
\end{tikzcd}$$

Consider the `quantum group law composition', $\rho\star \varphi=(\rho\otimes \varphi)\Delta$, with schematic:

$$
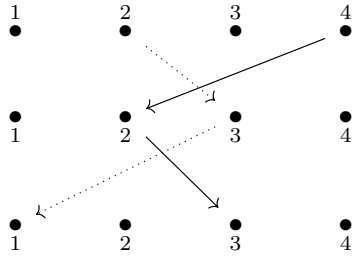
\begin{tikzcd}
\overset{1}{\bullet}  & \overset{2}{\bullet} \arrow[rd, dotted] & \overset{3}{\bullet}                      & \overset{4}{\bullet} \arrow[lld] \\
\underset{1}{\bullet} & \underset{2}{\bullet} \arrow[rd]        & \underset{3}{\bullet} \arrow[lld, dotted] & \underset{4}{\bullet}            \\
\underset{1}{\bullet} & \underset{2}{\bullet}                   & \underset{3}{\bullet}                     & \underset{4}{\bullet}
\end{tikzcd}$$
The solid lines represent certainties while the dashed lines denote (strictly) positive probabilities. Note that
\begin{align*}
  \mathbb{P}[(\rho\star\varphi)(4)=3]=(\rho\star \varphi)(u_{34}) & =(\rho\otimes \varphi)\left(\sum_k u_{3k}\otimes u_{k4}\right) \\
   & =\sum_k\rho(u_{3k})\varphi(u_{k4}) \\
   & =\rho(u_{32})\varphi(u_{24})=1.
\end{align*}
This is an example of a more general fact (see Proposition \ref{1.2}).

\bigskip

This calculation implies that
$$\widetilde{u_{34}}(\rho\star\varphi)=\rho\star\varphi.$$
Therefore to calculate the probability that $(\rho\star\varphi)(2)=1$ is observed after $(\rho\star\varphi)(3)=4$:
\begin{align*}
  (\rho\star \varphi)(u_{12}) & =(\rho\otimes \varphi)\left(\sum_k u_{1k}\otimes u_{k2}\right) \\
   & =\sum_k\rho(u_{1k})\varphi(u_{k2}) \\
   & =\rho(u_{13})\rho(u_{32})+\sum_{k\neq 3}\rho(u_{1k})\varphi(u_{k2}) >0.
\end{align*}
 It follows that
$$(\rho\star \varphi)(u_{34}u_{12}u_{34})>0\implies u_{12}u_{34}\neq 0.$$

Generalising this illustration gives:

\begin{theorem}
Let $\mathbb{G}\subseteq S_N^+$. The orbital relation is transitive.
\end{theorem}

\begin{proposition}\label{1.2}
  Suppose that $\varphi_1,\dots,\varphi_n$ are such that
  $$\varphi_1(u_{l_1j})=1,\,\varphi_2(u_{l_2l_1})=1,\,\varphi_3(u_{l_3l_2})=1,\dots\varphi_n(u_{il_{n-1}})=1.$$
  Then
  $$(\varphi_n\star\ldots \star\varphi_1)(u_{ij})=1.$$\end{proposition}
  \begin{proof}
  \begin{align*}
   (\varphi_n\star\ldots \star\varphi_1)(u_{ij})  & =\sum_{k_1,\dots,k_{n-1}}(\varphi_n\otimes\ldots\otimes \varphi_1)(u_{ik_{n-1}}u_{k_{n-1}k_{n-2}}\cdots u_{k_{1}j}) \\
    &= \varphi_n(u_{il_{n-1}})\varphi_{n-1}(u_{l_{n-1}l_{n-2}})\cdots\varphi_1(u_{l_1j})=1.
    \end{align*}
  \end{proof}
This can also be proven using the fact that the map $\Phi$ from the states to the doubly stochastic matrices is multiplicative.

\bigskip

The same approach does not work for three-orbitals. Suppose
$$u_{12}u_{33}u_{41}\neq0\text{ and }u_{22}u_{34}u_{11}\neq 0,$$
with quantum permutations $\rho_0$ and $\varphi_0$ such that:
$$\rho_0(|u_{12}u_{33}u_{41}|^2)>0\text{ and }\varphi_0(|u_{22}u_{34}u_{11}|^2)> 0.$$
Condition $\rho_0$ by $\rho_0(1)=4$ to $\rho$, and $\varphi_0$ by $\varphi_0(1)=1$ to $\varphi$, and take their product to find:
$$(\rho\star\varphi)(|u_{34}u_{41}|^2)>0.$$
However it does not follow that:
$$(\rho\star\varphi)(|u_{12}u_{34}u_{41}|^2)=\sum_{k_1,\dots,k_5}\rho(u_{4k_1}u_{3k_3}u_{1k_3}u_{3k_4}u_{4k_5})\varphi(u_{k_11}u_{k_24}u_{k_32}u_{k_44}u_{k_51})>0.$$
It is the case that
$$(k_1,k_2,k_3,k_4,k_5)=(1,3,2,3,1),$$
gives something strictly positive, but as monomials in the generators are not positive, cancellations  are possible.

\bigskip

Lupini, Man\v{c}inska, and Roberson \cite{lupini} (as well as  Banica \cite{quizzy}) expressed the belief that $\sim_3$ is not transitive in general.
The algebra of functions on a finite quantum group, as a finite dimensional $\mathrm{C}^*$-algebra, is a direct sum of $F(G_{\mathbb{G}})$, the direct sum of the one dimensional factors, and $B$, the direct sum of the higher dimensional factors. Counterexamples to $\sim_3$ transitive can occur in the finite case $\mathbb{G}\subset S_N^+$ when for the elements along the diagonal $u_{ii}F(\mathbb{G})\subseteq F(G_{\mathbb{G}})$; this implies that for all $\varphi\in\mathbb{G}$
$$\mathbb{P}[\varphi(i)=i]>0\implies \widetilde{u_{ii}}\varphi\text{ is a random permutation}.$$ If this is the case then for all $\varphi\in\mathbb{G}$
$$\mathbb{P}[[\varphi(i_1)\neq i_1]\succ [\varphi(i_2)=i_2]\succ [\varphi(i_1)=i_1]]=0,$$
which implies that for any $j_3\neq i_1$, $u_{j_3i_1}u_{i_2i_2}u_{i_1i_1}=0$. Therefore to find:   $$u_{i_1j_3}u_{i_2j_2}u_{i_1j_1}\neq 0 \text{, and }u_{j_3i_1}u_{j_2i_2}u_{j_1i_1}\neq 0,$$
yields the non-transitivity of $\sim_3$.

\bigskip

This phenomenon occurs in both the Kac-Paljutkin quantum group $\mathfrak{G}_0\subset S_4^+$ and also the dual $\widehat{Q}\subset S_8^+$ of the quaternions. In the case of $\mathfrak{G}_0\subset S_4^+$, the uncertainty phenomenon implies that the quantum permutation given by $\rho:=\widetilde{u_{41}}\varphi_{e_5}$ satisfies:
$$\mathbb{P}[[\rho(1)=3]\succ [\rho(3)=1]\succ [\rho(1)=4]]=\frac{1}{4}\implies u_{31}u_{13}u_{41}\neq 0\implies (3,1,4)\sim_3(1,3,1).$$
Similarly $\widetilde{u_{14}}\varphi_{e_5}$ shows that $u_{14}u_{31}u_{14}\neq 0$, and so $(1,3,1)\sim_3(4,1,4)$. For transitivity, it would have to be the case that $(3,1,4)\sim_3(4,1,4)$, that is $u_{34}u_{11}u_{44}\neq 0$, but as for all $\varphi'\in\mathfrak{G}_0$ such that $\mathbb{P}[\varphi'(4)=4]>0$, $\varphi:=\widetilde{u_{44}}\varphi'$ is a random permutation:
$$\mathbb{P}[[\varphi(4)=3]\succ [\varphi(1)=1]\succ [\varphi(4)=4]]=0\implies u_{34}u_{11}u_{44}=0,$$
and so $\sim_3$ is not transitive for $\mathfrak{G}_0\subset S_4^+$.

\subsection*{Acknowledgement}
I would like to thank the community of quantum group theorists who have helped and encouraged me the writing of this exposition. In particular I would like to thank Teo Banica.

\end{document}